\documentclass[12pt]{amsart}

\usepackage{amsmath,amssymb,amstext,amscd,amsthm,amsfonts}

\usepackage{pstricks}
\usepackage{mathrsfs}
\usepackage{fancyhdr}
\usepackage{a4wide}
\usepackage[numbers,sort]{natbib}
\usepackage{stmaryrd}
\usepackage[absolute,overlay]{textpos}
\usepackage{everyshi,calc,eso-pic,ifthen}
\usepackage{wallpaper}
\usepackage{dcolumn}
\usepackage{pdflscape}
\usepackage[all]{xy}
\usepackage[bookmarks]{hyperref}

\newtheorem{thm}{Theorem}[section]
\newtheorem{lem}[thm]{Lemma}

\newtheorem{cor}[thm]{Corollary}
\newtheorem{prop}[thm]{Proposition}

\theoremstyle{definition}
\newtheorem{dfn}[thm]{Definition}

\theoremstyle{remark}

\newtheorem{rem}[thm]{Remark}

\newcommand{\be}{\begin{eqnarray*}}
\newcommand{\ee}{\end{eqnarray*}}
\newcommand{\ba}{\begin{align*}}
\newcommand{\ea}{\end{align*}}

\newcommand{\Z}{\mathbb{Z}}
\newcommand{\Q}{\mathbb{Q}}

\newcommand{\CP}{\mathbb{C}\text{P}}

\newcommand{\x}{\times}
\newcommand{\xto}{\xrightarrow}
\newcommand{\hto}{\hookrightarrow}
\newcommand{\ot}{\otimes}

\newcommand{\im}{\text{im}\,}
\newcommand{\id}{\text{id}}

\newcommand{\ga}{\alpha}

\newcommand{\gc}{\gamma}

\numberwithin{equation}{section}

\begin{document}
\sloppy
\title[A smooth variation of BS Theory and PSC]{A smooth variation of Baas-Sullivan Theory and Positive Scalar Curvature}

\author{Sven F\"uhring}
\address{Department of Mathematics, University of Augsburg}
\email{sven.fuehring@uni-augsburg.de}
\date{March 19, 2012}

\begin{abstract}

Let $M$ be a smooth closed spin (resp.~ oriented and totally non-spin) manifold of dimension $n\geq 5$ with fundamental group $\pi$. It is stated, e.g.~ in \cite{rosto}, that $M$ admits a metric of positive scalar curvature (pscm) if its orientation class in $ko_n(B\pi)$ (resp.~ $H_n(B\pi;\Z)$) lies in the subgroup consisting of elements which contain pscm representatives. This is $2$-locally verified loc.~ cit.~ and in \cite{mspinsplit}. After inverting $2$ it was announced that a proof would be carried out in \cite{jung}, but this work has never appeared in print. The purpose of our paper is to present a self-contained proof of the statement with $2$ inverted.

\end{abstract}

\maketitle




\section{Introduction}

A basic question in Riemannian geometry is whether a given smooth closed manifold $M$ admits a metric of positive scalar curvature or not. Bordism theory is an important tool to approach this problem. On the one hand the surgery lemma (cf.~ \cite{gl},\cite{sy}) guarantees that under mild conditions the existence of a pscm is invariant under bordism. On the other hand, in case $M$ admits a spin structure, a certain characteristic class $\ga(M)$, again invariant under bordism, grants an obstruction to the existence of a pscm (cf.~ \cite{li},\cite{hitchin}).

Let $X$ be a space and $G=Spin$ or $SO$. As usual we denote the bordism groups of spin resp.~ oriented manifolds in $X$ by $\Omega_*^G(X)$. An element $[M,f]\in\Omega_n^G(X)$ is a bordism class of continuous maps $f\colon M\to X$ where $M$ is a smooth closed spin resp.~ oriented manifold of dimension $n$. We set
\[
{^+}\Omega_n^G(X):=\{[M,f]\in\Omega_n^{G}(X)\, |\, M\text{ admits a pscm}\}.
\]
In $\Omega_n^G(X)$ addition is given by the disjoint union of manifolds and taking inverses by reversing the spin structure resp.~ orientation. Hence ${^+}\Omega_n^G(X)$ in fact becomes a subgroup of $\Omega_n^G(X)$.
One can combine the surgery lemma with methods from the proof of the s-cobordism theorem to obtain the following existence result:
\begin{thm}[\cite{gl},\cite{rosensto}]\label{glr}
Let $M$ be a smooth connected closed manifold of dimension $n\geq 5$ with fundamental group $\pi$. Furthermore, let $B\pi$ be the classifying space of $\pi$ and $f\colon M\to B\pi$ the classifying map of the universal covering of $M$. Then the following holds:
\begin{enumerate}
\item If $M$ admits a spin structure it carries a pscm if and only if $[M,f]\in{^+}\Omega_n^{Spin}(B\pi)$.
\item If $M$ is orientable and totally non-spin, i.e.~ its universal cover does not admit a spin structure, it carries a pscm if and only if $[M,f]\in{^+}\Omega_n^{SO}(B\pi)$.
\end{enumerate}
\end{thm}

It is desirable to pass from the bordism groups of $B\pi$ to simpler groups which are easier to compute. In the oriented case we have the well-known map
\[
U\colon\Omega^{SO}_n(X)\to H_n(X;\Z)
\]
which sends an element $[M,f]$ to the image of the fundamental class $[M]\in H_n(M;\Z)$ under the induced map of $f$ in homology. Recall that in stable homotopy theory spectra determine homology theories and vice versa. The corresponding map in the spin case is the Atiyah-Bott-Shapiro orientation \cite{abss}
\[
A\colon\Omega^{Spin} _n(X)\to ko_n(X)
\]
where $ko_n(\_)$ denotes the homology theory associated to the connective cover of the real $K$-theory spectrum $KO$. We set $ko_n^+(X)=A({^+}\Omega_n^{Spin}(X))$ and $H_n^+(X)=U({^+}\Omega_n^{SO}(X))$. Theorem 4.11 in \cite{rosto} states:
\begin{thm}\label{sj}
Under the assumptions of Thm.~ \ref{glr} the following holds:
\begin{enumerate}
\item If $M$ admits a spin structure it carries a pscm if and only if $A([M,f])\in ko_n^+(B\pi)$.
\item If $M$ is orientable and totally non-spin it carries a pscm if and only if $U([M,f])\in H_n^+(B\pi)$.
\end{enumerate}
\end{thm}
To prove this statement one has to show
\begin{thm}\label{mt}
$\ker A\subset {^+}\Omega_*^{Spin}(B\pi)$ and $\ker U\subset {^+}\Omega_*^{SO}(B\pi)$.
\end{thm}
Such an inclusion of Abelian groups can be shown by proving it \emph{localized} at $2$, i.e.~ after tensoring with
$\Z_{(2)}:=\left\{\frac{a}{b}\in\Q\ |\  b\ \text{prime to}\ 2\right\}$ and after \emph{inverting} $2$, i.e.~ after tensoring with
$\Z \left[\frac{1}{2}\right]$.
\begin{enumerate}
\item $\ker A\ot\Z_{(2)}\subset {^+}\Omega_*^{Spin}(B\pi)\ot\Z_{(2)}$ is proved by Stolz \cite{mspinsplit} using splitting results of $MSpin$-module spectra.
\item $\ker U\ot\Z_{(2)}\subset {^+}\Omega_*^{SO}(B\pi)\ot\Z_{(2)}$ can be deduced from the Atiyah-Hirzebruch spectral sequence (sketched in \cite{rosto}).
\end{enumerate}
After inverting $2$ it is mentioned in \cite{rosto} that there is a proof by Rainer Jung \cite{jung}, for both the spin and the oriented case, based on the Baas-Sullivan theory of bordism with singularities. To the knowledge of the author, experts in this field agree that Jung's proof is probably correct. However, this proof is not available to the public (and in fact unknown to us). Hence one cannot verify its details, it is unclear how much technical effort is needed and generalizations or modifications cannot be carried out. Due to these reasons we shall fill this gap in the literature.

The strategy of our proof of Thm.~ \ref{mt} with $2$ inverted is as follows. Let $MSpin$ resp.~ $MSO$ denote the spin resp.~ oriented Thom spectrum and $H\Z$ the integer Eilenberg-MacLane spectrum. The orientation maps $A$ and $U$ are induced by spectrum maps $a\colon MSpin\to ko$ and $u\colon MSO\to H\Z$. We consider the fibrations
\begin{enumerate}\label{hf}
\item $\widehat{MSpin}\xto{i}MSpin\xto{a}ko$,
\item $\widehat{MSO}\xto{i}MSO\xto{u}H\Z$
\end{enumerate}
where $\widehat{MSpin}$ and $\widehat{MSO}$ denote the homotopy fibers of $a$ and $u$.

All groups and spectra are considered after inverting $2$. One can prove that the kernel of the induced map on coefficients, i.e.~ homotopy groups, $a_*\colon MSpin_*\to ko_*$ (resp.~ $u_*\colon MSO_*\to H\Z_*$) is generated by pscm manifolds, cf.~ \cite[Sec.~ 4]{hpeh} (resp.~ \cite{gl}). By means of these sets of pscm generators we shall give a geometric interpretation of the homology theories associated to $\widehat{MSpin}$ and $\widehat{MSO}$ in terms of smooth manifolds. It turns out that these manifolds also carry a pscm. Since $\ker A=\im (I\colon\widehat{MSpin}_*(X)\to MSpin_*(X))$ (and analogous in the oriented case) this proves Thm.~ \ref{mt} with $2$ inverted.\bigskip

\noindent\textbf{Overview.} Let $\mathscr{P}$ be a family of smooth closed manifolds. In section \ref{bt} of our paper we shall introduce a homology theory $\mathcal{P}_*(\_)$, which we call the \emph{bordism spanned by $\mathscr{P}$}. It is related to Baas-Sullivan theory as follows. Based on ideas of Sullivan \cite{sul} Baas \cite{baas1} introduced homology theories which provide a geometric description of singular homology by means of manifolds with singularities. Removing neighborhoods of these singularities Botvinnik \cite[Ch.~ 1]{bot} obtains a description by manifolds with additional structures on their boundaries. It is said in \cite{bot} that these boundaries theirselves lead to a homology theory, loc.~ cit.~ denoted by $MG_*^{\Sigma\Gamma(1)}(\_)$.

We cannot see obvious smooth structures on the manifolds which are used in the construction of $MG_*^{\Sigma\Gamma(1)}(\_)$. The theory $\mathcal{P}_*(\_)$ shall be a smooth variation of $MG_*^{\Sigma\Gamma(1)}(\_)$. Elements in $\mathcal{P}_*(\_)$ are represented by \emph{smooth} manifolds with additional structure. We directly verify that $\mathcal{P}_*(\_)$ satisfies the Eilenberg-Steenrod axioms. Afterwards we compute the coefficients $\mathcal{P}_*$ in our cases of interest and show how this leads to the description of the homotopy fibers $\widehat{MSpin}$ and $\widehat{MSO}$.

In section \ref{pscm} we shall prove our geometric result that the manifolds used in our description of $\mathcal{P}_*(\_)$ carry a pscm. With respect to constructing pscm, our description, which is based on smooth manifolds right away, seems to be more convenient than classical Baas-Sullivan theory and its further development by Botvinnik.

We note that our treatment of $\mathcal{P}_*(\_)$ is self-contained and can be considered as an alternative approach to Baas-Sullivan theory.\bigskip

\noindent\textbf{Acknowledgement.} I am indebted to my thesis advisor Bernhard Hanke for helpful guidance and continuous encouragement.


\section{A smooth Variation of Baas-Sullivan Theory}\label{bt}

We shall start with some preliminary remarks. Let $\mathbb{H}_i^n:=\{(x_1,\ldots,x_n)\in\mathbb{R}^n\, |\,x_i\geq0\}$. As usual a smooth $n$-dimensional manifold $M$ with boundary is modelled on $\mathbb{H}_n^n$.

We call subsets $N_1,\ldots,N_k$ of $M$ a \emph{transversal family of submanifolds} if for all $1\leq i_1<\ldots<i_l\leq k$ around every point in $N_{i_1}\cap\ldots\cap N_{i_l}$ there exists a chart $\psi\colon U\to\mathbb{H}^n_n$ of $M$ and an injective map $s\colon\{1,\ldots,l\}\to\{1,\ldots,n-1\}$ such that
\[
\psi(U\cap N_{i_j})=\psi(U)\cap\mathbb{H}^n_{s(j)}
\]
simultaneously for all $1\leq j\leq l$.

Let $M$ and $N$ denote smooth manifolds and let $A\subset M$ be a subset. A map $f\colon A\to N$ is called \emph{smooth} if $f$ is the restriction of a smooth map $M\to N$.

Finally, all upcoming manifolds are supposed to be oriented (resp.~ equipped with a spin structure) and we assume that all diffeomorphisms between manifolds preserve the orientation (resp.~ spin structure).
\bigskip

Now let $\mathscr{P}=\{P_1,P_2,\ldots,P_k\}$ be a finite family of smooth closed manifolds. For $I\subset\{1,\ldots,k\}$ we set $P_I:=\prod_{i\in I}P_i$.
\begin{dfn}\label{fdef} An $n$-dimensional \emph{$\mathscr{P}$-manifold} is a tuple $(M,(A_i)_{1\leq i\leq k},(B_I,\phi_I)_{I\subset\{1,\ldots,k\}})$ such that
\begin{itemize}
\item $M$ is a smooth $n$-dimensional manifold.
\item $A_1,\ldots,A_k$ is a transversal family of smooth $n$-dimensional submanifolds, closed as subsets, and the interiors of $A_i$ cover $M$.
\item For all $I\subset\{1,\ldots,k\}$, $B_I$ is a subset of some smooth manifold $C_I$ and $\phi_I$ is a map $A_I:=\cap_{i\in I} A_i\to P_I\x C_I$ which is a diffeomorphism onto $P_I\x B_I$.
\item For $J\subset I$ the map
\[
\phi_J\circ\phi_I^{-1}\colon P_J\x P_{I-J}\x B_I=P_I\x B_I\to P_J\x B_J
\]
is of the form $(x,y)\mapsto (x,\phi_J^I(y))$ where $x\in P_J$, $y\in P_{I-J}\x B_I$ and $\phi_J^I\colon P_{I-J}\x B_I\hto B_J$ is some map.
\end{itemize}
\end{dfn}

We agree that $\mathbb{H}_n^n$ always denotes the model space of the surrounding manifold $M$. Let us call $A_i\subset M$ the \emph{$P_i$-part} of a $\mathscr{P}$-manifold $M$. If all but one $B_i$ are empty we call $M$ a \emph{$P_i$-manifold}.

\begin{dfn}\label{mdd} Let $X$ be a space and $A\subset X$. An $n$-dimensional \emph{$\mathscr{P}$-manifold in $(X,A)$} is a tuple $(M,f,(A_i)_{1\leq i\leq k},(B_I,\phi_I)_{I\subset\{1,\ldots,k\}},(f_i)_{1\leq i\leq k})$ such that
\begin{itemize}
\item $(M,(A_i)_{1\leq i\leq k},(B_I,\phi_I)_{I\subset\{1,\ldots,k\}})$ is a compact $n$-dimensional $\mathscr{P}$-manifold.
\item $f\colon(M,\partial M)\to(X,A)$ and $f_i\colon B_i\to X$ are continuous maps such that the diagram
\[
\xymatrix{
A_i\ar[r]^-f\ar[d]^-{\phi_i}  & X\\
P_i\x B_i\ar[r]^-{pr}      & B_i.\ar[u]_-{f_i}
}
\]
commutes for all $i$.
\end{itemize}
\end{dfn}
In the sequel we fix a family $\mathscr{P}=\{P_1,P_2,\ldots,P_k\}$ of smooth closed manifolds and write $(M,f,A_i,B_I,\phi_I,f_i)$, $(M,f,A_i)$ or $(M,f)$ short for a $\mathscr{P}$-manifold in $(X,A)$.

\begin{dfn}\label{pbor}
An $n$-dimensional $\mathscr{P}$-manifold $(M,F,A_i,B_I,\phi_I,f_i)$ in $(X,A)$ is said to \emph{$\mathscr{P}$-bord} if there exists a tuple $(\hat{M},\hat{f} ,\hat{A},\hat{B}_I,\hat{\phi}_I,\hat{f}_i)$ such that
\begin{itemize}
\item $(\hat{M},\hat{A}_i,\hat{B}_I,\hat{\phi}_I)$ is a compact $(n+1)$-dimensional $\mathscr{P}$-manifold.
\item $M\subset\partial\hat{M}$ and for all $i$ one has $\hat{A}_i\cap M=A_i$. In addition, each
\[
\hat{\phi}_i\circ\phi_i^{-1}\colon P_i\x B_i\to P_i\x \hat{B}_i
\]
is of the form $(x,y)\mapsto(x,\omega_i(y))$ for some map $\omega_i$.
\item $\hat{F}\colon\hat{M}\to X$ and $\hat{f}_i\colon\hat{B}_i\to X$ are continuous maps such that $\hat{F}(\partial\hat{M}-M)\subset A$, $\hat{f}|_M=f$ and the diagram
\[
\xymatrix{
\hat{A}_i\ar[r]^-{\hat{F}}\ar[d]^-{\hat{\phi}_i}  & X\\
P_i\x \hat{B}_i\ar[r]^-{pr}      & \hat{B}_i.\ar[u]_-{\hat{f}_i}
}
\]
commutes for all $i$.
\end{itemize}
\end{dfn}

One continues as in ordinary bordism homology. The disjoint union of two $\mathscr{P}$-manifolds in $(X,A)$ is again a $\mathscr{P}$-manifold. We say that two $n$-dimensional $\mathscr{P}$-manifolds $(M,f)$ and $(N,g)$ in $(X,A)$ are \emph{$\mathscr{P}$-bordant} if $(M,f)\, \dot{\cup}\, (-N,g)$ $\mathscr{P}$-bords.

\begin{lem}\label{er} $\mathscr{P}$-bordism defines an equivalence relation. \end{lem}
\begin{proof}
Let $(M,f,A_i,B_I,\phi_I,f_i)$ be $\mathscr{P}$-manifold in $(X,A)$. For the proof of reflexivity we consider
\begin{equation}\label{ref}
(M\x[0,1],f\circ pr,A_i\x[0,1],B_I\x[0,1],\phi_I\circ pr,f_i\circ pr).
\end{equation}
By 'straightening the angle' \cite[p.~ 8]{conner}, $M\x[0,1]$ can be given a differentiable structure. By doing so, $A_1\x[0,1],\ldots, A_k\x[0,1]$ becomes a transversal family of submanifolds. There is an induced straightening of $B_I$ for all $I$ and thus \ref{ref} becomes a $\mathscr{P}$-bordism between $(M,f)$ and itself.

The symmetry relation is obvious.

To prove transitivity let $(M,f)$ and $(N,g)$ resp.~ $(N,g)$ and $(O,h)$ be $\mathscr{P}$-bordant $n$-dimensional $\mathscr{P}$-manifolds in $(X,A)$, say via $(V,F,A_i,B_I,\phi_I,f_i)$ resp.~ $(W,G,C_i,D_I,\psi_I,g_i)$. Because of transversality one finds charts of $V$ around $A_I\cap\partial V$ and of $W$ around $C_I\cap\partial W$ in which the respective inner boundaries $\overline{\partial A_i-\partial V}$ and $\overline{\partial C_i-\partial W}$ of the $P_i$-parts lie on a common $\partial\mathbb{H}_j^{n+1}$ for some $j\leq n$ depending on $i$. Hence, for all $i$ we can glue $A_i$ and $C_i$ along $(A_i\cap\partial V)|_N\cong (C_i\cap\partial W)|_N$ such that $A_1\cup C_1,\ldots, A_k\cup C_k$ becomes a transversal family of submanifolds of the resulting smooth manifold $V\cup W$. Let the $P_i$-part of $N$ be diffeomorphic to $P_i\x E_i$. By means of point two of Def. \ref{pbor} one recovers $E_i$ as a submanifold of $B_i$ and $D_i$. Thus, for all $i$ we can also glue $B_i$ and $D_i$ along $E_i$. One obtains an induces bonding of $B_I$ and $D_I$ for all $I$. Now the desired $\mathscr{P}$-bordism between $(M,f)$ and $(O,h)$ is given by
\[
(V\cup W, F\cup G, A_i\cup C_i, B_I\cup D_I, \phi_I\cup\psi_I,f_i\cup g_i).
\]
\end{proof}

Denote by $\mathcal{P}_n(X,A)$ the set of all $\mathscr{P}$-bordism classes of $n$-dimensional $\mathscr{P}$-manifolds in $(X,A)$. Via disjoint union it becomes an Abelian group with zero element the $\mathscr{P}$-bordism class which $\mathscr{P}$-bords. A map $g\colon (X,A)\to (Y,B)$ induces a group homomorphism $\mathcal{P}_n(X,A)\to \mathcal{P}_n(Y,B)$ by $[(M,f)]\mapsto[(M,g\circ f)]$. If $(M,f)$ is a $\mathscr{P}$-manifold in $(X,A)$ then the boundary of $M$ becomes a $\mathscr{P}$-manifold in $A$ by restriction. It is denoted by $\partial(M,f)$. Then we have an induced map $\partial\colon\mathcal{P}_n(X,A)\to\mathcal{P}_{n-1} (A,\emptyset)$ defined by $[(M,f)]\mapsto [\partial(M,f)]$.

\begin{prop}\label{homt}  The \emph{bordism spanned by $\mathscr{P}$}
\[
\mathcal{P}_*(X,A):=\bigoplus_{n\geq0}\mathcal{P}_n(X,A)
\]
is a homology theory. \end{prop}

\begin{proof} We have to show that $\mathcal{P}_*(X,A)$ satisfies the Eilenberg-Steenrod axioms. One proceeds in the same way as in the case of ordinary bordism homology \cite[p.~ 11-13]{conner} and additionally takes the local product structures into account. However, the proof of excision requires special attention. For the sake of completeness we shall verify all axioms.

Let $i\colon A\hto X$ and $j\colon(X,\emptyset)\hto (X,A)$ denote the inclusions. Obviously $\mathcal{P}_*(\_)$ is a functor from the category of pairs of topological spaces (with continuous maps as morphisms) to the category of Abelian groups. It remains to show:

\textbf{Homotopy axiom.} Let $g,\,  h\colon (X,A)\to (Y,B)$ be homotopic via $H\colon(X,A)\x[0,1]\to(Y,B)$. Then $g_*=h_*\colon \mathcal{P}_n(X,A)\to\mathcal{P}_n(Y,B)$.

Let $(M,f,A_i,B_I,\phi_I,f_i)$ be a $\mathscr{P}$-manifold in $(X,A)$. We define
\[
G\colon M\x[0,1]\to Y, (x,t)\mapsto H(f(x),t).
\]
By straightening the angle $M\x[0,1]$ can be equipped with the structure of a $\mathscr{P}$-manifold. Then
\[
(M\x[0,1],G,A_i\x[0,1],B_I\x[0,1],\phi_I\x\id,H\circ(f_i\x\id))
\]
becomes a $\mathscr{P}$-bordism between $(M,gf)$ and $(M,hf)$.

\textbf{Exactness axiom.} The sequence
\[
\ldots\xto{\partial}\mathcal{P}_{n} (A)\xto{i_*}\mathcal{P}_n(X)\xto{j_*}\mathcal{P}_n(X,A)\xto{\partial}\mathcal{P}_{n-1} (A)\xto{i_*}\ldots
\]
is exact.

It is clear that $\partial j_*=0$ and $i_* \partial=0$. Let $[(M,f)]\in\mathcal{P}_n(A)$, a zero $\mathscr{P}$-bordism for $(M,j\, i\, f)$ is given by $(M\x[0,1],j\, i\, f\, pr)$, hence $j_*i_*=0$.

Let $[(M,f,A_i,B_I,\phi_I,f_i)]\in\mathcal{P}_n(X,A)$ be in the kernel of $\partial$. Then $\partial(M,f)$ bords in $A$, i.e.~ there exists a zero $\mathscr{P}$-bordism $(W,g,C_i,D_I,\psi_I,g_i)$ for $\partial(M,F)$ in $A$. As in the proof of transitivity in Lemma \ref{er} we can glue $A_i$ and $C_i$ along $A_i\cap\partial M\cong C_i\cap\partial W$ for all $i$ to obtain a closed $\mathscr{P}$-manifold $N$ and a map $(f\cup g)\colon N\to X$.
Now $(N\x[0,1], (f\cup g)\circ pr)$ is a $\mathscr{P}$-bordism between $(N,j\circ(f\cup g))$ and $(M,f)$ in $(X,A)$. The cases $\ker j_*\subset\im i_*$ and $\ker i_*\subset\im\partial$ are obvious.

\textbf{Excision axiom.} Let $U$ be an open subset of $X$ such that $\overline{U}\subset\mathring{A}$, then the inclusion $i\colon(X-U,A-U)\hto (X,A)$ induces an isomorphism
\[
i_*\colon\mathcal{P}_n(X-U,A-U)\xto{\cong}\mathcal{P}_n(X,A).
\]

First we show that $i_*$ is epic: Let $(M,f,A_i,B_I,\phi_I,f_i)$ be a $\mathscr{P}$-manifold in $(X,A)$. We are looking for a smooth submanifold $N\subset M$ such that $f^{-1}(X-\mathring{A})\subset N$ and $f^{-1}(\overline{U})\cap N=\emptyset$. In addition, $N$ shall respect the local product structures in the sense that $\phi_I(A_I\cap N)=P_I\x C_I$ for some $C_I\subset B_I$. Then it follows that $N$ inherits a $\mathscr{P}$-structure of $M$. Now $(N,F|_N)$ defines an element in $\mathcal{P}_n(X-U,A-U)$ and
\[
(M\x[0,1],f\circ pr, A_i\x[0,1],B_I\x[0,1],\phi_I\x\id, f_i\circ pr)
\]
is a $\mathscr{P}$-bordism between $(N,i\circ f|_N)$ and $(M,f)$ in $(X,A)$.

The construction of $N$ requires a preliminary observation. Until the end of this proof we shall denote the 'inner' boundary $\overline{\partial A_i-\partial M}$ by $\partial A_i$, likewise for comparable sets. If $\partial A_i\x[0,1]$ is a collar neighborhood for $\partial A_i\subset A_i$, say $\partial A_i\x\{0\}=\partial A_i$, we set
\[
A_i^{t}:=A_i-(\partial A_i\x[0,t))
\]
for all $0\leq t\leq1$ (see figure \ref{aus}). Now one observes that there exists collar neighborhoods $\partial A_i\x[0,1]$ for all $i$ such that for arbitrary sequences of numbers $0\leq t_1,\ldots,t_k\leq 1$ and by suitable restrictions $(M,f,A_i^{t_i})$ become $\mathscr{P}$-bordant $\mathscr{P}$-manifolds in $(X,A)$.

Now let us construct $N$. We set $Q:=f^{-1}(X-\mathring{A})$ and $R:=f_I^{-1}(\overline{U})$. One can show \cite[Lemma 3.1]{conner} that there exists an $n$-dimensional submanifold $N_0\subset M$, closed as a subset, such that $Q\subset N_0$ and $R\cap N_0=\emptyset$. Clearly, $N_0$ does not have to respect the local product structures. Therefore we shall modify $N_0$ as follows. Set $B'_i:=pr_{B_i}(\phi_i(N_0\cap A_i))$ and consider the \emph{saturation of the $P_i$-fibers}
\[
N_1:=\bigcup_{i=1}^k\left(\bigcup_{b\in B'_i} \phi_i^{-1}\left(P_i\x\{b\}\right)\right).
\]
Due to the condition $\cap_{i\in I}A_i\cong P_I\x B_I$ this $N_1$ respects the local product structures. In addition, as $f$ locally factors over $B_I$ one concludes $f(N_0)=f(N_1)$. Now $N_1$ is the union of manifolds modelled on $\mathbb{H}_1\cap\mathbb{H}_2$. The non-smooth points of $N_1$ only occur on $C_1:=\cup_{i=1}^k\partial A_i$. With respect to the metric induced by the collar neighborhoods let $U_1$ be an open $1/k$-neighborhood of $C_1$. One finds a smooth submanifold $N'_1\subset M$ such that $N'_1-U_1=N_1-U_1$. In view of continuity $N'_1$ can be chosen such that any longer $Q\subset N'_1$ and $R\cap N'_1=\emptyset$. Note that $N'_1$ respects the local product structures except on $U_1$.

Now we replace $A_i$ by $A^{1/k}_i$ for all $i$ and repeat the above procedure. The saturations of the $P_i$-fibers with respect to the $A^{1/k}_i$'s yield an $N_2\subset M$. One merely has to perform this saturation step on $U_1$ and, again, non-smooth points only appear on $\cup_{i=1}^k\partial A^{1/k} _i$. Hence the non-smooth points of $N_2$ occur on an open $1/k$-neighborhood $U_2$ of
\begin{align*}
C_2:&=C_1\cap\left(\bigcup_{i=1}^k\partial A^{1/k}_i\right)\\
    &=\bigcup_{1\leq i,j\leq k}\left(\partial A_i \cap\partial A^{1/k}_j\right).
\end{align*}
Similar as above there is a smooth submanifold $N'_2\subset M$ which satisfies $N'_2-U_2=N_2-U_2$ and $Q\subset N'_2$, $R\cap N'_2=\emptyset$.

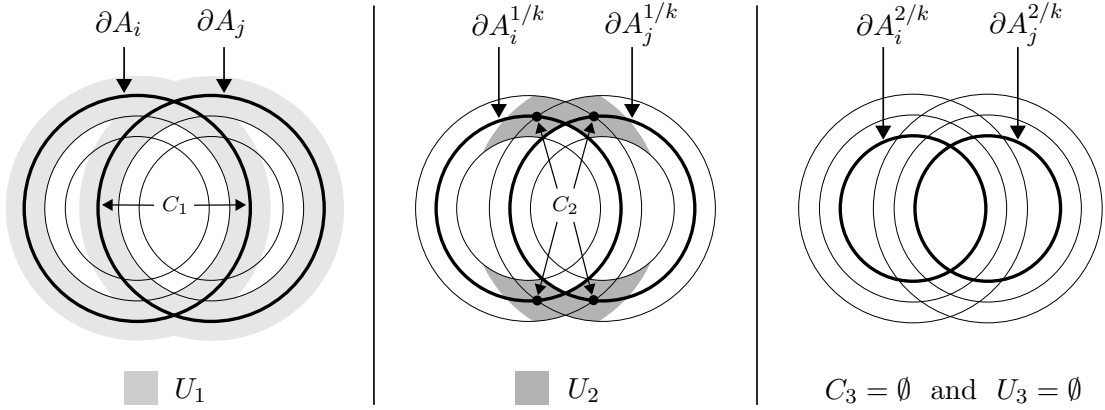
\begin{figure}
\psset{xunit=.63pt,yunit=.63pt,runit=.63pt}
\begin{pspicture}(71,700)(665,980)
{
\newrgbcolor{curcolor}{0.90196079 0.90196079 0.90196079}
\pscustom[linestyle=none,fillstyle=solid,fillcolor=curcolor]
{
\newpath
\moveto(154.78086,753.53810262)
\curveto(127.08054,754.46271262)(101.79615,769.95422262)(88.038939,794.43026262)
\curveto(82.875379,803.61696262)(79.445769,814.62767262)(78.382049,825.43359262)
\curveto(78.034209,828.96708262)(78.034209,837.11119262)(78.382049,840.64468262)
\curveto(81.924589,876.63174262)(109.22994,905.77919262)(144.8056,911.54935262)
\curveto(151.37007,912.61407262)(160.0071,912.83113262)(166.39762,912.09199262)
\curveto(200.76602,908.11688262)(228.54932,882.76738262)(235.46452,849.07530262)
\curveto(236.63308,843.38188262)(237.03296,839.29334262)(237.03296,833.03913262)
\curveto(237.03296,826.78492262)(236.63308,822.69639262)(235.46452,817.00296262)
\curveto(232.32671,801.71501262)(224.90492,787.97866262)(213.76208,776.83583262)
\curveto(198.09178,761.16553262)(176.91181,752.79939262)(154.78086,753.53810262)
\closepath
\moveto(162.10447,777.61166262)
\curveto(193.00885,780.46817262)(215.58866,807.87853262)(212.47462,838.75762262)
\curveto(210.45217,858.81246262)(197.56714,876.37945262)(179.09278,884.26926262)
\curveto(171.92322,887.33114262)(165.03301,888.72379262)(157.0537,888.72379262)
\curveto(147.32095,888.72379262)(138.33706,886.44016262)(130.09014,881.86988262)
\curveto(104.10365,867.46869262)(93.973889,835.19787262)(107.06295,808.51121262)
\curveto(117.21383,787.81507262)(139.16597,775.49144262)(162.10447,777.61166262)
\closepath
}
}
{
\newrgbcolor{curcolor}{0.90196079 0.90196079 0.90196079}
\pscustom[linestyle=none,fillstyle=solid,fillcolor=curcolor]
{
\newpath
\moveto(110.73287,753.72607262)
\curveto(83.032554,754.65068262)(57.748164,770.14219262)(43.990954,794.61823262)
\curveto(38.827394,803.80493262)(35.397784,814.81564262)(34.334064,825.62156262)
\curveto(33.986224,829.15505262)(33.986224,837.29916262)(34.334064,840.83265262)
\curveto(37.876604,876.81971262)(65.181954,905.96716262)(100.75761,911.73732262)
\curveto(107.32208,912.80204262)(115.95911,913.01910262)(122.34963,912.27996262)
\curveto(156.71803,908.30485262)(184.50133,882.95535262)(191.41653,849.26328262)
\curveto(192.58509,843.56985262)(192.98497,839.48131262)(192.98497,833.22710262)
\curveto(192.98497,826.97289262)(192.58509,822.88436262)(191.41653,817.19093262)
\curveto(188.27872,801.90298262)(180.85693,788.16663262)(169.71409,777.02380262)
\curveto(154.04379,761.35350262)(132.86382,752.98736262)(110.73287,753.72607262)
\closepath
\moveto(118.05648,777.79963262)
\curveto(148.96086,780.65614262)(171.54067,808.06650262)(168.42663,838.94559262)
\curveto(166.40418,859.00043262)(153.51915,876.56742262)(135.04479,884.45723262)
\curveto(127.87523,887.51911262)(120.98502,888.91176262)(113.00571,888.91176262)
\curveto(103.27296,888.91176262)(94.289074,886.62813262)(86.042154,882.05785262)
\curveto(60.055664,867.65666262)(49.925904,835.38584262)(63.014964,808.69918262)
\curveto(73.165844,788.00304262)(95.117984,775.67941262)(118.05648,777.79963262)
\closepath
}
}
{
\newrgbcolor{curcolor}{0 0 0}
\pscustom[linewidth=1,linecolor=curcolor]
{
\newpath
\moveto(255,954.99999362)
\lineto(255,715.00000262)
}
}
{
\newrgbcolor{curcolor}{0 0 0}
\pscustom[linewidth=2.00000014,linecolor=curcolor]
{
\newpath
\moveto(180.800195,833.04979933)
\curveto(180.800195,795.54974904)(150.40039779,765.14995183)(112.9003475,765.14995183)
\curveto(75.40029721,765.14995183)(45.0005,795.54974904)(45.0005,833.04979933)
\curveto(45.0005,870.54984962)(75.40029721,900.94964683)(112.9003475,900.94964683)
\curveto(150.40039779,900.94964683)(180.800195,870.54984962)(180.800195,833.04979933)
\closepath
}
}
{
\newrgbcolor{curcolor}{0 0 0}
\pscustom[linewidth=0.5,linecolor=curcolor]
{
\newpath
\moveto(168.454635,833.04969614)
\curveto(168.454635,802.36781096)(143.58205269,777.49522864)(112.9001675,777.49522864)
\curveto(82.21828231,777.49522864)(57.3457,802.36781096)(57.3457,833.04969614)
\curveto(57.3457,863.73158133)(82.21828231,888.60416364)(112.9001675,888.60416364)
\curveto(143.58205269,888.60416364)(168.454635,863.73158133)(168.454635,833.04969614)
\closepath
}
}
{
\newrgbcolor{curcolor}{0 0 0}
\pscustom[linewidth=0.50000007,linecolor=curcolor]
{
\newpath
\moveto(156.109265,833.04969993)
\curveto(156.109265,809.18601022)(136.7639222,789.84066743)(112.9002325,789.84066743)
\curveto(89.0365428,789.84066743)(69.6912,809.18601022)(69.6912,833.04969993)
\curveto(69.6912,856.91338963)(89.0365428,876.25873243)(112.9002325,876.25873243)
\curveto(136.7639222,876.25873243)(156.109265,856.91338963)(156.109265,833.04969993)
\closepath
}
}
{
\newrgbcolor{curcolor}{0 0 0}
\pscustom[linewidth=2.00000014,linecolor=curcolor]
{
\newpath
\moveto(225.131495,833.04979933)
\curveto(225.131495,795.54974904)(194.73169779,765.14995183)(157.2316475,765.14995183)
\curveto(119.73159721,765.14995183)(89.3318,795.54974904)(89.3318,833.04979933)
\curveto(89.3318,870.54984962)(119.73159721,900.94964683)(157.2316475,900.94964683)
\curveto(194.73169779,900.94964683)(225.131495,870.54984962)(225.131495,833.04979933)
\closepath
}
}
{
\newrgbcolor{curcolor}{0 0 0}
\pscustom[linewidth=0.5,linecolor=curcolor]
{
\newpath
\moveto(212.786035,833.04959614)
\curveto(212.786035,802.36771096)(187.91345269,777.49512864)(157.2315675,777.49512864)
\curveto(126.54968231,777.49512864)(101.6771,802.36771096)(101.6771,833.04959614)
\curveto(101.6771,863.73148133)(126.54968231,888.60406364)(157.2315675,888.60406364)
\curveto(187.91345269,888.60406364)(212.786035,863.73148133)(212.786035,833.04959614)
\closepath
}
}
{
\newrgbcolor{curcolor}{0 0 0}
\pscustom[linewidth=0.50000007,linecolor=curcolor]
{
\newpath
\moveto(200.440565,833.04969993)
\curveto(200.440565,809.18601022)(181.0952222,789.84066743)(157.2315325,789.84066743)
\curveto(133.3678428,789.84066743)(114.0225,809.18601022)(114.0225,833.04969993)
\curveto(114.0225,856.91338963)(133.3678428,876.25873243)(157.2315325,876.25873243)
\curveto(181.0952222,876.25873243)(200.440565,856.91338963)(200.440565,833.04969993)
\closepath
}
}
{
\newrgbcolor{curcolor}{0 0 0}
\pscustom[linewidth=0.50000006,linecolor=curcolor]
{
\newpath
\moveto(647.384345,832.99992503)
\curveto(647.384345,795.06231334)(616.62983419,764.30780253)(578.6922225,764.30780253)
\curveto(540.75461081,764.30780253)(510.0001,795.06231334)(510.0001,832.99992503)
\curveto(510.0001,870.93753672)(540.75461081,901.69204753)(578.6922225,901.69204753)
\curveto(616.62983419,901.69204753)(647.384345,870.93753672)(647.384345,832.99992503)
\closepath
}
}
{
\newrgbcolor{curcolor}{0 0 0}
\pscustom[linewidth=0.50000004,linecolor=curcolor]
{
\newpath
\moveto(634.894685,832.99978081)
\curveto(634.894685,801.95989084)(609.73188247,776.79708831)(578.6919925,776.79708831)
\curveto(547.65210253,776.79708831)(522.4893,801.95989084)(522.4893,832.99978081)
\curveto(522.4893,864.03967078)(547.65210253,889.20247331)(578.6919925,889.20247331)
\curveto(609.73188247,889.20247331)(634.894685,864.03967078)(634.894685,832.99978081)
\closepath
}
}
{
\newrgbcolor{curcolor}{0 0 0}
\pscustom[linewidth=2.00000006,linecolor=curcolor]
{
\newpath
\moveto(622.405215,832.99984356)
\curveto(622.405215,808.85770569)(602.83414537,789.28663606)(578.6920075,789.28663606)
\curveto(554.54986963,789.28663606)(534.9788,808.85770569)(534.9788,832.99984356)
\curveto(534.9788,857.14198143)(554.54986963,876.71305106)(578.6920075,876.71305106)
\curveto(602.83414537,876.71305106)(622.405215,857.14198143)(622.405215,832.99984356)
\closepath
}
}
{
\newrgbcolor{curcolor}{0 0 0}
\pscustom[linewidth=0.50000006,linecolor=curcolor]
{
\newpath
\moveto(692.232945,832.99992503)
\curveto(692.232945,795.06231334)(661.47843419,764.30780253)(623.5408225,764.30780253)
\curveto(585.60321081,764.30780253)(554.8487,795.06231334)(554.8487,832.99992503)
\curveto(554.8487,870.93753672)(585.60321081,901.69204753)(623.5408225,901.69204753)
\curveto(661.47843419,901.69204753)(692.232945,870.93753672)(692.232945,832.99992503)
\closepath
}
}
{
\newrgbcolor{curcolor}{0 0 0}
\pscustom[linewidth=0.50000004,linecolor=curcolor]
{
\newpath
\moveto(679.743285,832.99978081)
\curveto(679.743285,801.95989084)(654.58048247,776.79708831)(623.5405925,776.79708831)
\curveto(592.50070253,776.79708831)(567.3379,801.95989084)(567.3379,832.99978081)
\curveto(567.3379,864.03967078)(592.50070253,889.20247331)(623.5405925,889.20247331)
\curveto(654.58048247,889.20247331)(679.743285,864.03967078)(679.743285,832.99978081)
\closepath
}
}
{
\newrgbcolor{curcolor}{0 0 0}
\pscustom[linewidth=2.00000006,linecolor=curcolor]
{
\newpath
\moveto(667.253815,832.99984356)
\curveto(667.253815,808.85770569)(647.68274537,789.28663606)(623.5406075,789.28663606)
\curveto(599.39846963,789.28663606)(579.8274,808.85770569)(579.8274,832.99984356)
\curveto(579.8274,857.14198143)(599.39846963,876.71305106)(623.5406075,876.71305106)
\curveto(647.68274537,876.71305106)(667.253815,857.14198143)(667.253815,832.99984356)
\closepath
}
}
{
\newrgbcolor{curcolor}{0 0 0}
\pscustom[linewidth=1,linecolor=curcolor]
{
\newpath
\moveto(485,955.00000062)
\lineto(485,715.00000262)
}
}
{
\newrgbcolor{curcolor}{0.7019608 0.7019608 0.7019608}
\pscustom[linestyle=none,fillstyle=solid,fillcolor=curcolor]
{
\newpath
\moveto(334.10364,887.11970262)
\curveto(334.29213,887.46642262)(338.53593,891.59955262)(340.11851,892.97773262)
\curveto(342.91129,895.40978262)(346.5018,898.12222262)(349.28405,899.90180262)
\lineto(350.51208,900.68727262)
\lineto(352.28405,900.59967262)
\curveto(356.67633,900.38250262)(361.97587,899.45881262)(366.68101,898.09033262)
\curveto(369.68853,897.21560262)(369.67076,897.41889262)(366.83905,896.28296262)
\curveto(362.31979,894.47008262)(358.36447,892.39852262)(354.4081,889.77238262)
\lineto(352.38518,888.42962262)
\lineto(350.2206,888.55104262)
\curveto(345.70522,888.80431262)(339.28737,888.20725262)(335.09769,887.14413262)
\curveto(333.93411,886.84888262)(333.95671,886.84944262)(334.10364,887.11973262)
\closepath
}
}
{
\newrgbcolor{curcolor}{0.7019608 0.7019608 0.7019608}
\pscustom[linestyle=none,fillstyle=solid,fillcolor=curcolor]
{
\newpath
\moveto(368.30601,884.72424262)
\curveto(364.26684,886.35194262)(359.07226,887.68077262)(354.86851,888.16172262)
\curveto(354.14664,888.24432262)(353.48139,888.35063262)(353.39018,888.39800262)
\curveto(353.09388,888.55187262)(357.06519,891.07205262)(360.48633,892.90119262)
\curveto(363.76559,894.65447262)(369.15705,896.94856262)(369.99351,896.94654262)
\curveto(370.87158,896.94454262)(376.54514,894.52083262)(379.73676,892.78450262)
\curveto(382.45858,891.30375262)(386.69567,888.68644262)(386.57169,888.56246262)
\curveto(386.54759,888.53836262)(385.69221,888.40691262)(384.67094,888.27044262)
\curveto(380.10159,887.65985262)(375.71915,886.52742262)(371.57407,884.88620262)
\curveto(370.55274,884.48181262)(369.70899,884.15527262)(369.69907,884.16055262)
\curveto(369.68907,884.16555262)(369.06226,884.41949262)(368.30601,884.72424262)
\lineto(368.30601,884.72424262)
\closepath
}
}
{
\newrgbcolor{curcolor}{0.7019608 0.7019608 0.7019608}
\pscustom[linestyle=none,fillstyle=solid,fillcolor=curcolor]
{
\newpath
\moveto(339.49351,875.58725262)
\curveto(339.49351,875.78844262)(342.49074,879.12768262)(344.25376,880.89071262)
\curveto(346.82507,883.46201262)(351.71113,887.60679262)(352.70464,888.05946262)
\curveto(353.25884,888.31197262)(358.15749,887.53285262)(361.5881,886.64656262)
\curveto(364.28563,885.94966262)(368.54537,884.51993262)(369.28369,884.06362262)
\curveto(369.34259,884.02722262)(368.5561,883.59200262)(367.53591,883.09649262)
\curveto(363.88194,881.32174262)(360.35417,879.08566262)(357.18101,876.53300262)
\curveto(356.17967,875.72747262)(355.86885,875.53986262)(355.61851,875.58988262)
\curveto(351.22953,876.46688262)(344.90595,876.46172262)(339.96226,875.57708262)
\curveto(339.70445,875.53098262)(339.49351,875.53548262)(339.49351,875.58718262)
\closepath
}
}
{
\newrgbcolor{curcolor}{0.7019608 0.7019608 0.7019608}
\pscustom[linestyle=none,fillstyle=solid,fillcolor=curcolor]
{
\newpath
\moveto(320.24351,866.45441262)
\curveto(320.24351,866.63394262)(322.48406,870.98783262)(323.33581,872.46345262)
\curveto(325.98551,877.05392262)(329.18545,881.55835262)(332.41436,885.24296262)
\curveto(333.03106,885.94669262)(333.62459,886.54790262)(333.73332,886.57897262)
\curveto(338.418,887.91791262)(346.16321,888.68379262)(351.08529,888.29479262)
\lineto(352.11457,888.21349262)
\lineto(350.64779,887.09196262)
\curveto(346.65425,884.03838262)(341.88665,879.40361262)(339.18101,875.94463262)
\curveto(338.74011,875.38096262)(338.59976,875.28775262)(337.99351,875.15595262)
\curveto(334.55476,874.40834262)(329.64218,872.47843262)(326.31646,870.56862262)
\curveto(324.99903,869.81208262)(322.98182,868.47210262)(321.79104,867.56251262)
\curveto(321.26973,867.16430262)(320.70827,866.73537262)(320.54336,866.60932262)
\curveto(320.37844,866.48327262)(320.24351,866.41358262)(320.24351,866.45446262)
\closepath
}
}
{
\newrgbcolor{curcolor}{0.7019608 0.7019608 0.7019608}
\pscustom[linestyle=none,fillstyle=solid,fillcolor=curcolor]
{
\newpath
\moveto(348.1846,766.98078262)
\curveto(345.46171,768.83183262)(343.07465,770.63789262)(340.83318,772.54289262)
\curveto(339.22458,773.91003262)(334.16513,778.81981262)(334.26353,778.91821262)
\curveto(334.29333,778.94801262)(335.26959,778.77838262)(336.43299,778.54126262)
\curveto(340.66018,777.67966262)(342.34654,777.52036262)(347.69412,777.47746262)
\lineto(352.77412,777.43676262)
\lineto(354.6768,776.20927262)
\curveto(358.99232,773.42518262)(362.47723,771.62035262)(367.06112,769.79543262)
\lineto(369.58779,768.78952262)
\lineto(367.32628,768.12325262)
\curveto(362.9807,766.84298262)(358.47881,765.97038262)(354.10462,765.56051262)
\curveto(350.58692,765.23089262)(350.85249,765.16717262)(348.1846,766.98082262)
\closepath
}
}
{
\newrgbcolor{curcolor}{0.7019608 0.7019608 0.7019608}
\pscustom[linestyle=none,fillstyle=solid,fillcolor=curcolor]
{
\newpath
\moveto(345.089,777.69913262)
\curveto(342.21294,777.82542262)(336.50725,778.68464262)(334.21668,779.33641262)
\curveto(333.55622,779.52434262)(331.77752,781.48956262)(329.11344,784.97479262)
\curveto(327.58912,786.96895262)(324.54434,791.54837262)(323.39606,793.57384262)
\curveto(322.44441,795.25249262)(320.39434,799.31909262)(320.49974,799.31909262)
\curveto(320.54034,799.31909262)(321.29644,798.77624262)(322.17998,798.11276262)
\curveto(326.31353,795.00871262)(331.4635,792.53851262)(336.66178,791.16653262)
\lineto(338.81156,790.59914262)
\lineto(340.26813,788.90451262)
\curveto(343.41178,785.24707262)(346.84529,781.98637262)(350.52557,779.16334262)
\lineto(352.24982,777.84072262)
\lineto(350.61396,777.72792262)
\curveto(348.93578,777.61221262)(347.26677,777.60351262)(345.089,777.69912262)
\lineto(345.089,777.69912262)
\closepath
}
}
{
\newrgbcolor{curcolor}{0.7019608 0.7019608 0.7019608}
\pscustom[linestyle=none,fillstyle=solid,fillcolor=curcolor]
{
\newpath
\moveto(352.14727,778.32266262)
\curveto(349.62393,780.19001262)(345.80876,783.52078262)(343.65623,785.73563262)
\curveto(342.06428,787.37367262)(339.4827,790.32456262)(339.56746,790.40932262)
\curveto(339.59746,790.43932262)(340.29522,790.34422262)(341.11806,790.19796262)
\curveto(344.37286,789.61952262)(351.98345,789.68538262)(355.19855,790.31980262)
\lineto(356.11571,790.50078262)
\lineto(357.44341,789.42434262)
\curveto(360.2643,787.13728262)(365.84772,783.69789262)(368.67284,782.50699262)
\lineto(369.62925,782.10383262)
\lineto(368.87764,781.78462262)
\curveto(364.75575,780.03404262)(357.89966,778.28940262)(354.0798,778.01906262)
\curveto(352.82569,777.93026262)(352.63752,777.95986262)(352.14727,778.32266262)
\closepath
}
}
{
\newrgbcolor{curcolor}{0.7019608 0.7019608 0.7019608}
\pscustom[linestyle=none,fillstyle=solid,fillcolor=curcolor]
{
\newpath
\moveto(367.71642,769.85533262)
\curveto(363.24364,771.58574262)(359.24209,773.62928262)(355.35323,776.16904262)
\curveto(352.93571,777.74789262)(352.82071,777.51807262)(356.27659,778.01422262)
\curveto(360.11732,778.56561262)(364.65309,779.77501262)(368.14612,781.17906262)
\curveto(369.81462,781.84973262)(370.10225,781.91723262)(370.47411,781.72536262)
\curveto(373.08211,780.37977262)(380.30487,778.43764262)(384.61903,777.92194262)
\curveto(385.53145,777.81287262)(386.39299,777.68683262)(386.53356,777.64186262)
\curveto(387.05454,777.47517262)(380.92019,773.73757262)(377.2484,771.98450262)
\curveto(374.36082,770.60584262)(370.58084,769.08817262)(370.05582,769.09665262)
\curveto(369.83865,769.10065262)(368.78592,769.44157262)(367.71642,769.85533262)
\closepath
}
}
{
\newrgbcolor{curcolor}{0.7019608 0.7019608 0.7019608}
\pscustom[linestyle=none,fillstyle=solid,fillcolor=curcolor]
{
\newpath
\moveto(407.48342,886.40928262)
\curveto(407.37517,886.51752262)(404.91197,887.12020262)(403.10429,887.48073262)
\curveto(398.91712,888.31582262)(393.69574,888.72487262)(389.87249,888.51732262)
\lineto(387.80999,888.40536262)
\lineto(387.15961,888.87905262)
\curveto(383.35595,891.64942262)(377.75504,894.65718262)(372.6877,896.65065262)
\curveto(370.67594,897.44207262)(370.63511,897.30981262)(373.11787,898.04415262)
\curveto(376.08603,898.92205262)(380.00363,899.75221262)(383.16993,900.17423262)
\curveto(385.0895,900.43008262)(388.61328,900.72504262)(389.77304,900.72695262)
\lineto(390.73609,900.72895262)
\lineto(392.01553,899.89869262)
\curveto(394.83659,898.06801262)(398.12972,895.56991262)(400.87249,893.17998262)
\curveto(403.39502,890.98196262)(407.99433,886.35395262)(407.65619,886.35395262)
\curveto(407.59179,886.35395262)(407.51407,886.37905262)(407.48342,886.40965262)
\closepath
}
}
{
\newrgbcolor{curcolor}{0.7019608 0.7019608 0.7019608}
\pscustom[linestyle=none,fillstyle=solid,fillcolor=curcolor]
{
\newpath
\moveto(420.68499,865.56062262)
\curveto(419.50348,866.63232262)(416.78374,868.64841262)(415.03488,869.74895262)
\curveto(413.05277,870.99628262)(409.48803,872.74860262)(407.2349,873.58321262)
\curveto(405.69726,874.15278262)(402.06185,875.22854262)(401.67468,875.22854262)
\curveto(401.58878,875.22854262)(401.33099,875.46760262)(401.10181,875.75979262)
\curveto(398.24106,879.40704262)(393.66438,883.90493262)(389.58274,887.08053262)
\curveto(388.77025,887.71266262)(388.17963,888.25299262)(388.27024,888.28127262)
\curveto(388.71563,888.42023262)(392.5916,888.46130262)(394.80999,888.35057262)
\curveto(398.51157,888.16580262)(401.91475,887.68003262)(405.24749,886.86070262)
\curveto(407.96885,886.19168262)(407.96352,886.19423262)(408.88029,885.12382262)
\curveto(412.02185,881.45574262)(415.04082,877.19027262)(417.52058,872.91603262)
\curveto(418.94007,870.46932262)(421.62465,865.05857262)(421.48447,864.92683262)
\curveto(421.45727,864.90123262)(421.09749,865.18646262)(420.68499,865.56062262)
\lineto(420.68499,865.56062262)
\closepath
}
}
{
\newrgbcolor{curcolor}{0.7019608 0.7019608 0.7019608}
\pscustom[linestyle=none,fillstyle=solid,fillcolor=curcolor]
{
\newpath
\moveto(399.74749,875.60869262)
\curveto(399.50686,875.66619262)(398.57874,875.81512262)(397.68499,875.93964262)
\curveto(395.41124,876.25643262)(388.73729,876.25835262)(386.62249,875.94264262)
\curveto(385.83186,875.82469262)(384.95984,875.69455262)(384.68465,875.65345262)
\curveto(384.21539,875.58335262)(384.12999,875.62345262)(383.30965,876.29923262)
\curveto(380.00644,879.02022262)(376.20246,881.45602262)(372.25383,883.37860262)
\lineto(370.69768,884.13629262)
\lineto(371.94133,884.62071262)
\curveto(376.16563,886.26617262)(380.53985,887.37692262)(385.1802,887.98249262)
\lineto(387.05041,888.22656262)
\lineto(387.5552,887.89577262)
\curveto(389.18716,886.82637262)(392.75592,883.81406262)(395.20626,881.43768262)
\curveto(396.62215,880.06453262)(399.49802,876.93795262)(400.37153,875.82211262)
\curveto(400.61884,875.50619262)(400.62219,875.47939262)(400.41283,875.49115262)
\curveto(400.28753,875.49815262)(399.98813,875.55095262)(399.74751,875.60850262)
\lineto(399.74751,875.60850262)
\closepath
}
}
{
\newrgbcolor{curcolor}{0.7019608 0.7019608 0.7019608}
\pscustom[linestyle=none,fillstyle=solid,fillcolor=curcolor]
{
\newpath
\moveto(385.12249,778.00210262)
\curveto(380.18712,778.71097262)(375.57253,779.90846262)(371.58798,781.51433262)
\lineto(370.24096,782.05721262)
\lineto(371.52548,782.64832262)
\curveto(375.13066,784.30736262)(379.58619,787.10711262)(382.85171,789.76545262)
\lineto(383.76844,790.51172262)
\lineto(385.03921,790.30862262)
\curveto(387.82201,789.86385262)(388.76134,789.79654262)(392.18499,789.79654262)
\curveto(395.60917,789.79654262)(396.63122,789.87034262)(399.22636,790.30471262)
\curveto(399.86787,790.41210262)(400.41468,790.47802262)(400.4415,790.45120262)
\curveto(400.5332,790.35950262)(397.67001,787.16850262)(395.94942,785.44470262)
\curveto(393.40445,782.89495262)(390.4693,780.37136262)(387.92756,778.54764262)
\curveto(386.86501,777.78525262)(386.76613,777.76602262)(385.12249,778.00210262)
\lineto(385.12249,778.00210262)
\closepath
}
}
{
\newrgbcolor{curcolor}{0.7019608 0.7019608 0.7019608}
\pscustom[linestyle=none,fillstyle=solid,fillcolor=curcolor]
{
\newpath
\moveto(388.05999,777.71101262)
\curveto(387.64749,777.75501262)(387.39722,777.82717262)(387.50385,777.87132262)
\curveto(387.82173,778.00295262)(390.5657,780.17084262)(392.33672,781.68955262)
\curveto(394.46507,783.51469262)(397.80953,786.87846262)(399.55999,788.95454262)
\lineto(400.93499,790.58532262)
\lineto(402.55999,790.98064262)
\curveto(408.66827,792.46664262)(414.96502,795.63761262)(419.68666,799.60543262)
\curveto(421.11713,800.80753262)(421.17409,800.84307262)(421.06075,800.46295262)
\curveto(420.90189,799.93015262)(418.6161,795.45839262)(417.58644,793.66604262)
\curveto(415.0225,789.20295262)(411.92244,784.84546262)(408.58302,781.01070262)
\curveto(407.50214,779.76950262)(407.44272,779.72302262)(406.63902,779.48998262)
\curveto(404.6119,778.90220262)(401.24215,778.25686262)(398.37249,777.90686262)
\curveto(396.59557,777.69014262)(389.5156,777.55568262)(388.05999,777.71101262)
\lineto(388.05999,777.71101262)
\closepath
}
}
{
\newrgbcolor{curcolor}{0.7019608 0.7019608 0.7019608}
\pscustom[linestyle=none,fillstyle=solid,fillcolor=curcolor]
{
\newpath
\moveto(387.74749,765.36985262)
\curveto(382.72594,765.70931262)(377.32819,766.68500262)(372.68499,768.09254262)
\curveto(370.63948,768.71261262)(370.2518,768.86388262)(370.49749,768.94607262)
\curveto(370.60061,768.98057262)(371.52874,769.33548262)(372.55999,769.73476262)
\curveto(377.06836,771.48034262)(381.39491,773.73694262)(385.55999,776.51517262)
\lineto(387.05999,777.51572262)
\lineto(391.99749,777.53582262)
\curveto(396.29451,777.55332262)(397.19442,777.59112262)(398.93499,777.82754262)
\curveto(401.31896,778.15133262)(403.05061,778.46413262)(404.99749,778.92268262)
\curveto(407.01984,779.39901262)(406.99749,779.39516262)(406.99749,779.26724262)
\curveto(406.99749,779.09192262)(402.50389,774.69679262)(400.87249,773.27647262)
\curveto(396.90342,769.82093262)(390.43002,765.20034262)(389.62249,765.24642262)
\curveto(389.45061,765.25642262)(388.60686,765.31182262)(387.74749,765.36987262)
\closepath
}
}
{
\newrgbcolor{curcolor}{0 0 0}
\pscustom[linewidth=0.99999994,linecolor=curcolor]
{
\newpath
\moveto(105,930.00000262)
\lineto(105,908.00000262)
}
}
{
\newrgbcolor{curcolor}{0 0 0}
\pscustom[linestyle=none,fillstyle=solid,fillcolor=curcolor]
{
\newpath
\moveto(105,903.38400289)
\lineto(101.00000024,910.30400248)
\lineto(108.99999976,910.30400248)
\lineto(105,903.38400289)
\closepath
}
}
{
\newrgbcolor{curcolor}{0 0 0}
\pscustom[linewidth=0.99999994,linecolor=curcolor]
{
\newpath
\moveto(105,903.38400289)
\lineto(101.00000024,910.30400248)
\lineto(108.99999976,910.30400248)
\lineto(105,903.38400289)
\closepath
}
}
{
\newrgbcolor{curcolor}{0 0 0}
\pscustom[linewidth=1,linecolor=curcolor]
{
\newpath
\moveto(165,930.00000262)
\lineto(165,908.00002262)
}
}
{
\newrgbcolor{curcolor}{0 0 0}
\pscustom[linestyle=none,fillstyle=solid,fillcolor=curcolor]
{
\newpath
\moveto(165,903.38402262)
\lineto(161,910.30402262)
\lineto(169,910.30402262)
\lineto(165,903.38402262)
\closepath
}
}
{
\newrgbcolor{curcolor}{0 0 0}
\pscustom[linewidth=1,linecolor=curcolor]
{
\newpath
\moveto(165,903.38402262)
\lineto(161,910.30402262)
\lineto(169,910.30402262)
\lineto(165,903.38402262)
\closepath
}
}
{
\newrgbcolor{curcolor}{0 0 0}
\pscustom[linewidth=1,linecolor=curcolor]
{
\newpath
\moveto(330,930.00000262)
\lineto(330,893.00002262)
}
}
{
\newrgbcolor{curcolor}{0 0 0}
\pscustom[linestyle=none,fillstyle=solid,fillcolor=curcolor]
{
\newpath
\moveto(330,888.38402262)
\lineto(326,895.30402262)
\lineto(334,895.30402262)
\lineto(330,888.38402262)
\closepath
}
}
{
\newrgbcolor{curcolor}{0 0 0}
\pscustom[linewidth=1,linecolor=curcolor]
{
\newpath
\moveto(330,888.38402262)
\lineto(326,895.30402262)
\lineto(334,895.30402262)
\lineto(330,888.38402262)
\closepath
}
}
{
\newrgbcolor{curcolor}{0 0 0}
\pscustom[linewidth=0.75000006,linecolor=curcolor]
{
\newpath
\moveto(365,845.00001262)
\lineto(354.86995,880.76003262)
}
}
{
\newrgbcolor{curcolor}{0 0 0}
\pscustom[linestyle=none,fillstyle=solid,fillcolor=curcolor]
{
\newpath
\moveto(353.92636845,884.09096338)
\lineto(358.22734492,879.91511531)
\lineto(352.45450131,878.279792)
\lineto(353.92636845,884.09096338)
\closepath
}
}
{
\newrgbcolor{curcolor}{0 0 0}
\pscustom[linewidth=0.75000006,linecolor=curcolor]
{
\newpath
\moveto(353.92636845,884.09096338)
\lineto(358.22734492,879.91511531)
\lineto(352.45450131,878.279792)
\lineto(353.92636845,884.09096338)
\closepath
}
}
{
\newrgbcolor{curcolor}{0 0 0}
\pscustom[linewidth=0.75,linecolor=curcolor]
{
\newpath
\moveto(374.99998,845.00000262)
\lineto(384.81294,880.85916262)
}
}
{
\newrgbcolor{curcolor}{0 0 0}
\pscustom[linestyle=none,fillstyle=solid,fillcolor=curcolor]
{
\newpath
\moveto(385.72672875,884.19838919)
\lineto(387.25044802,878.40059827)
\lineto(381.46322692,879.98428761)
\lineto(385.72672875,884.19838919)
\closepath
}
}
{
\newrgbcolor{curcolor}{0 0 0}
\pscustom[linewidth=0.75,linecolor=curcolor]
{
\newpath
\moveto(385.72672875,884.19838919)
\lineto(387.25044802,878.40059827)
\lineto(381.46322692,879.98428761)
\lineto(385.72672875,884.19838919)
\closepath
}
}
{
\newrgbcolor{curcolor}{0 0 0}
\pscustom[linewidth=0.75,linecolor=curcolor]
{
\newpath
\moveto(365,825.00000262)
\lineto(354.58188,785.05267262)
}
}
{
\newrgbcolor{curcolor}{0 0 0}
\pscustom[linestyle=none,fillstyle=solid,fillcolor=curcolor]
{
\newpath
\moveto(353.70822487,781.70272177)
\lineto(352.11504682,787.48181198)
\lineto(357.92085418,785.9676783)
\lineto(353.70822487,781.70272177)
\closepath
}
}
{
\newrgbcolor{curcolor}{0 0 0}
\pscustom[linewidth=0.75,linecolor=curcolor]
{
\newpath
\moveto(353.70822487,781.70272177)
\lineto(352.11504682,787.48181198)
\lineto(357.92085418,785.9676783)
\lineto(353.70822487,781.70272177)
\closepath
}
}
{
\newrgbcolor{curcolor}{0 0 0}
\pscustom[linewidth=0.74999994,linecolor=curcolor]
{
\newpath
\moveto(375,825.00000262)
\lineto(384.92807,785.19407262)
}
}
{
\newrgbcolor{curcolor}{0 0 0}
\pscustom[linestyle=none,fillstyle=solid,fillcolor=curcolor]
{
\newpath
\moveto(385.76586865,781.83497519)
\lineto(381.59906701,786.14471514)
\lineto(387.42072633,787.59670586)
\lineto(385.76586865,781.83497519)
\closepath
}
}
{
\newrgbcolor{curcolor}{0 0 0}
\pscustom[linewidth=0.74999994,linecolor=curcolor]
{
\newpath
\moveto(385.76586865,781.83497519)
\lineto(381.59906701,786.14471514)
\lineto(387.42072633,787.59670586)
\lineto(385.76586865,781.83497519)
\closepath
}
}
{
\newrgbcolor{curcolor}{0 0 0}
\pscustom[linewidth=0.75000006,linecolor=curcolor]
{
\newpath
\moveto(123,835.00000262)
\lineto(96,835.15790262)
}
}
{
\newrgbcolor{curcolor}{0 0 0}
\pscustom[linestyle=none,fillstyle=solid,fillcolor=curcolor]
{
\newpath
\moveto(92.53805892,835.17814856)
\lineto(97.74551473,838.14774613)
\lineto(97.71042644,832.14784825)
\lineto(92.53805892,835.17814856)
\closepath
}
}
{
\newrgbcolor{curcolor}{0 0 0}
\pscustom[linewidth=0.75000006,linecolor=curcolor]
{
\newpath
\moveto(92.53805892,835.17814856)
\lineto(97.74551473,838.14774613)
\lineto(97.71042644,832.14784825)
\lineto(92.53805892,835.17814856)
\closepath
}
}
{
\newrgbcolor{curcolor}{0 0 0}
\pscustom[linewidth=0.75,linecolor=curcolor]
{
\newpath
\moveto(147,835.00000262)
\lineto(174,835.15790262)
}
}
{
\newrgbcolor{curcolor}{0 0 0}
\pscustom[linestyle=none,fillstyle=solid,fillcolor=curcolor]
{
\newpath
\moveto(177.4619408,835.17814856)
\lineto(172.28957369,832.14784849)
\lineto(172.2544854,838.14774589)
\lineto(177.4619408,835.17814856)
\closepath
}
}
{
\newrgbcolor{curcolor}{0 0 0}
\pscustom[linewidth=0.75,linecolor=curcolor]
{
\newpath
\moveto(177.4619408,835.17814856)
\lineto(172.28957369,832.14784849)
\lineto(172.2544854,838.14774589)
\lineto(177.4619408,835.17814856)
\closepath
}
}
{
\newrgbcolor{curcolor}{0.80000001 0.80000001 0.80000001}
\pscustom[linestyle=none,fillstyle=solid,fillcolor=curcolor]
{
\newpath
\moveto(105,735)
\lineto(125,735)
\lineto(125,715)
\lineto(105,715)
\closepath
}
}
{
\newrgbcolor{curcolor}{0.7019608 0.7019608 0.7019608}
\pscustom[linestyle=none,fillstyle=solid,fillcolor=curcolor]
{
\newpath
\moveto(340,735)
\lineto(360,735)
\lineto(360,715)
\lineto(340,715)
\closepath
}
}
{
\newrgbcolor{curcolor}{0 0 0}
\pscustom[linewidth=0.99999958,linecolor=curcolor]
{
\newpath
\moveto(560.8,930.00000262)
\lineto(560.8,880.00000262)
}
}
{
\newrgbcolor{curcolor}{0 0 0}
\pscustom[linestyle=none,fillstyle=solid,fillcolor=curcolor]
{
\newpath
\moveto(560.8,875.38400454)
\lineto(556.80000167,882.30400166)
\lineto(564.79999833,882.30400166)
\lineto(560.8,875.38400454)
\closepath
}
}
{
\newrgbcolor{curcolor}{0 0 0}
\pscustom[linewidth=0.99999958,linecolor=curcolor]
{
\newpath
\moveto(560.8,875.38400454)
\lineto(556.80000167,882.30400166)
\lineto(564.79999833,882.30400166)
\lineto(560.8,875.38400454)
\closepath
}
}
{
\newrgbcolor{curcolor}{0 0 0}
\pscustom[linewidth=1,linecolor=curcolor]
{
\newpath
\moveto(410,930.00000262)
\lineto(410,893.00002262)
}
}
{
\newrgbcolor{curcolor}{0 0 0}
\pscustom[linestyle=none,fillstyle=solid,fillcolor=curcolor]
{
\newpath
\moveto(410,888.38402262)
\lineto(406,895.30402262)
\lineto(414,895.30402262)
\lineto(410,888.38402262)
\closepath
}
}
{
\newrgbcolor{curcolor}{0 0 0}
\pscustom[linewidth=1,linecolor=curcolor]
{
\newpath
\moveto(410,888.38402262)
\lineto(406,895.30402262)
\lineto(414,895.30402262)
\lineto(410,888.38402262)
\closepath
}
}
{
\newrgbcolor{curcolor}{0 0 0}
\pscustom[linewidth=0.99999958,linecolor=curcolor]
{
\newpath
\moveto(641.6,930.00000262)
\lineto(641.6,880.00000262)
}
}
{
\newrgbcolor{curcolor}{0 0 0}
\pscustom[linestyle=none,fillstyle=solid,fillcolor=curcolor]
{
\newpath
\moveto(641.6,875.38400454)
\lineto(637.60000167,882.30400166)
\lineto(645.59999833,882.30400166)
\lineto(641.6,875.38400454)
\closepath
}
}
{
\newrgbcolor{curcolor}{0 0 0}
\pscustom[linewidth=0.99999958,linecolor=curcolor]
{
\newpath
\moveto(641.6,875.38400454)
\lineto(637.60000167,882.30400166)
\lineto(645.59999833,882.30400166)
\lineto(641.6,875.38400454)
\closepath
}
}
{
\newrgbcolor{curcolor}{0.7019608 0.7019608 0.7019608}
\pscustom[linestyle=none,fillstyle=solid,fillcolor=curcolor]
{
\newpath
\moveto(401.15006,875.50098262)
\curveto(401.06726,875.51958262)(401.03298,875.55418262)(400.75944,875.89410262)
\curveto(399.25008,877.76985262)(397.29791,879.90896262)(395.42126,881.74345262)
\curveto(393.61631,883.50785262)(391.39235,885.44343262)(389.50793,886.88999262)
\lineto(389.13666,887.17499262)
\lineto(389.18406,887.26149262)
\curveto(389.22446,887.33539262)(389.23636,887.34369262)(389.26516,887.31889262)
\curveto(389.28366,887.30289262)(389.54998,887.09053262)(389.8569,886.84703262)
\curveto(393.78518,883.73057262)(397.99007,879.61376262)(400.73045,876.20125262)
\curveto(400.87833,876.01710262)(401.07184,875.77854262)(401.16047,875.67112262)
\curveto(401.33005,875.46558262)(401.32975,875.46057262)(401.15007,875.50101262)
\closepath
}
}
{
\newrgbcolor{curcolor}{0.7019608 0.7019608 0.7019608}
\pscustom[linestyle=none,fillstyle=solid,fillcolor=curcolor]
{
\newpath
\moveto(384.84025,890.19811262)
\curveto(384.52402,890.41526262)(383.21598,891.22591262)(382.4826,891.65924262)
\curveto(379.61408,893.35416262)(376.66656,894.81843262)(373.60141,896.07124262)
\curveto(372.79996,896.39882262)(371.65211,896.84234262)(371.10345,897.03645262)
\lineto(370.80415,897.14233262)
\lineto(371.01265,897.21223262)
\curveto(371.19434,897.27313262)(371.22566,897.27803262)(371.25625,897.25043262)
\curveto(371.32225,897.19073262)(371.60687,897.06271262)(372.1522,896.84750262)
\curveto(374.65868,895.85833262)(376.32793,895.11481262)(378.5456,893.99975262)
\curveto(380.27014,893.13263262)(381.74307,892.32517262)(383.27754,891.40569262)
\curveto(383.93118,891.01401262)(384.77821,890.48809262)(385.03589,890.31392262)
\lineto(385.16457,890.22692262)
\lineto(385.10577,890.14982262)
\curveto(385.07347,890.10742262)(385.04067,890.07332262)(385.03287,890.07412262)
\curveto(385.02487,890.07486262)(384.93847,890.13062262)(384.84021,890.19809262)
\lineto(384.84021,890.19809262)
\closepath
}
}
{
\newrgbcolor{curcolor}{0.7019608 0.7019608 0.7019608}
\pscustom[linestyle=none,fillstyle=solid,fillcolor=curcolor]
{
\newpath
\moveto(338.92278,875.62423262)
\curveto(338.98328,875.68873262)(339.20304,875.95235262)(339.41105,876.21016262)
\curveto(341.17843,878.40064262)(343.29621,880.62652262)(346.01366,883.14977262)
\curveto(346.1577,883.28352262)(346.28063,883.38787262)(346.28684,883.38166262)
\curveto(346.29284,883.37566262)(346.01911,883.11106262)(345.67808,882.79412262)
\curveto(344.8776,882.05016262)(343.0787,880.25009262)(342.32175,879.43561262)
\curveto(341.28342,878.31837262)(339.99951,876.84333262)(339.28916,875.95160262)
\curveto(338.97978,875.56323262)(338.92546,875.50704262)(338.85936,875.50704262)
\curveto(338.82016,875.50704262)(338.83016,875.52564262)(338.92276,875.62423262)
\lineto(338.92276,875.62423262)
\closepath
}
}
{
\newrgbcolor{curcolor}{0.7019608 0.7019608 0.7019608}
\pscustom[linestyle=none,fillstyle=solid,fillcolor=curcolor]
{
\newpath
\moveto(346.46429,883.55897262)
\curveto(346.65516,883.73881262)(347.48719,884.47552262)(347.95983,884.88318262)
\curveto(348.85282,885.65340262)(349.85145,886.46560262)(350.71094,887.12070262)
\curveto(350.73024,887.13540262)(350.73884,887.13190262)(350.73884,887.10940262)
\curveto(350.73884,887.09150262)(350.59571,886.96775262)(350.42076,886.83443262)
\curveto(349.54804,886.16932262)(347.99152,884.87618262)(346.95031,883.95120262)
\curveto(346.62833,883.66516262)(346.35713,883.43150262)(346.34763,883.43195262)
\curveto(346.33763,883.43240262)(346.39063,883.48955262)(346.46429,883.55896262)
\closepath
}
}
{
\newrgbcolor{curcolor}{0.7019608 0.7019608 0.7019608}
\pscustom[linestyle=none,fillstyle=solid,fillcolor=curcolor]
{
\newpath
\moveto(367.93451,896.69878262)
\curveto(367.93451,896.70378262)(368.04249,896.75038262)(368.17446,896.80290262)
\curveto(368.50993,896.93630262)(368.80397,897.08471262)(368.88437,897.16120262)
\lineto(368.95257,897.22600262)
\lineto(369.07239,897.18530262)
\lineto(369.19221,897.14460262)
\lineto(368.90274,897.04076262)
\curveto(368.74354,896.98366262)(368.46841,896.88137262)(368.29136,896.81350262)
\curveto(367.98206,896.69493262)(367.93455,896.67965262)(367.93455,896.69879262)
\closepath
}
}
{
\newrgbcolor{curcolor}{0.7019608 0.7019608 0.7019608}
\pscustom[linestyle=none,fillstyle=solid,fillcolor=curcolor]
{
\newpath
\moveto(354.76706,889.97400262)
\lineto(354.75016,889.99090262)
\lineto(354.93692,890.11333262)
\curveto(355.73042,890.63350262)(356.81409,891.30340262)(357.72009,891.83384262)
\curveto(360.10314,893.22905262)(362.63109,894.50001262)(365.24493,895.61706262)
\curveto(365.41253,895.68866262)(365.45178,895.70266262)(365.47972,895.70076262)
\curveto(365.51052,895.69876262)(365.49102,895.68886262)(365.23902,895.58073262)
\curveto(363.56898,894.86385262)(361.74223,893.98408262)(360.08763,893.09981262)
\curveto(358.40841,892.20237262)(356.73367,891.21399262)(355.14185,890.18096262)
\curveto(354.95218,890.05787262)(354.79405,889.95716262)(354.79045,889.95716262)
\curveto(354.78645,889.95716262)(354.77625,889.96516262)(354.76705,889.97406262)
\closepath
}
}
{
\newrgbcolor{curcolor}{0 0 0}
\pscustom[linewidth=0.50000004,linecolor=curcolor]
{
\newpath
\moveto(415.70094,832.99992954)
\curveto(415.70094,795.5271773)(385.32327224,765.14950954)(347.85052,765.14950954)
\curveto(310.37776776,765.14950954)(280.0001,795.5271773)(280.0001,832.99992954)
\curveto(280.0001,870.47268178)(310.37776776,900.85034954)(347.85052,900.85034954)
\curveto(385.32327224,900.85034954)(415.70094,870.47268178)(415.70094,832.99992954)
\closepath
}
}
{
\newrgbcolor{curcolor}{0 0 0}
\pscustom[linewidth=2.00000003,linecolor=curcolor]
{
\newpath
\moveto(403.364455,832.99974681)
\curveto(403.364455,802.34019602)(378.50997829,777.48571931)(347.8504275,777.48571931)
\curveto(317.19087671,777.48571931)(292.3364,802.34019602)(292.3364,832.99974681)
\curveto(292.3364,863.6592976)(317.19087671,888.51377431)(347.8504275,888.51377431)
\curveto(378.50997829,888.51377431)(403.364455,863.6592976)(403.364455,832.99974681)
\closepath
}
}
{
\newrgbcolor{curcolor}{0 0 0}
\pscustom[linewidth=0.49999999,linecolor=curcolor]
{
\newpath
\moveto(391.027855,832.99997591)
\curveto(391.027855,809.15365832)(371.69659509,789.82239841)(347.8502775,789.82239841)
\curveto(324.00395991,789.82239841)(304.6727,809.15365832)(304.6727,832.99997591)
\curveto(304.6727,856.8462935)(324.00395991,876.17755341)(347.8502775,876.17755341)
\curveto(371.69659509,876.17755341)(391.027855,856.8462935)(391.027855,832.99997591)
\closepath
}
}
{
\newrgbcolor{curcolor}{0 0 0}
\pscustom[linewidth=0.50000004,linecolor=curcolor]
{
\newpath
\moveto(460.00004,832.99992954)
\curveto(460.00004,795.5271773)(429.62237224,765.14950954)(392.14962,765.14950954)
\curveto(354.67686776,765.14950954)(324.2992,795.5271773)(324.2992,832.99992954)
\curveto(324.2992,870.47268178)(354.67686776,900.85034954)(392.14962,900.85034954)
\curveto(429.62237224,900.85034954)(460.00004,870.47268178)(460.00004,832.99992954)
\closepath
}
}
{
\newrgbcolor{curcolor}{0 0 0}
\pscustom[linewidth=2.00000003,linecolor=curcolor]
{
\newpath
\moveto(447.663555,832.99964681)
\curveto(447.663555,802.34009602)(422.80907829,777.48561931)(392.1495275,777.48561931)
\curveto(361.48997671,777.48561931)(336.6355,802.34009602)(336.6355,832.99964681)
\curveto(336.6355,863.6591976)(361.48997671,888.51367431)(392.1495275,888.51367431)
\curveto(422.80907829,888.51367431)(447.663555,863.6591976)(447.663555,832.99964681)
\closepath
}
}
{
\newrgbcolor{curcolor}{0 0 0}
\pscustom[linewidth=0.49999999,linecolor=curcolor]
{
\newpath
\moveto(435.326955,832.99987591)
\curveto(435.326955,809.15355832)(415.99569509,789.82229841)(392.1493775,789.82229841)
\curveto(368.30305991,789.82229841)(348.9718,809.15355832)(348.9718,832.99987591)
\curveto(348.9718,856.8461935)(368.30305991,876.17745341)(392.1493775,876.17745341)
\curveto(415.99569509,876.17745341)(435.326955,856.8461935)(435.326955,832.99987591)
\closepath
}
}
{
\newrgbcolor{curcolor}{0 0 0}
\pscustom[linestyle=none,fillstyle=solid,fillcolor=curcolor]
{
\newpath
\moveto(389.866071,777.95536)
\curveto(389.866071,776.29850575)(388.52292525,774.95536)(386.866071,774.95536)
\curveto(385.20921675,774.95536)(383.866071,776.29850575)(383.866071,777.95536)
\curveto(383.866071,779.61221425)(385.20921675,780.95536)(386.866071,780.95536)
\curveto(388.52292525,780.95536)(389.866071,779.61221425)(389.866071,777.95536)
\closepath
}
}
{
\newrgbcolor{curcolor}{0 0 0}
\pscustom[linestyle=none,fillstyle=solid,fillcolor=curcolor]
{
\newpath
\moveto(356.044643,777.73215)
\curveto(356.044643,776.07529575)(354.70149725,774.73215)(353.044643,774.73215)
\curveto(351.38778875,774.73215)(350.044643,776.07529575)(350.044643,777.73215)
\curveto(350.044643,779.38900425)(351.38778875,780.73215)(353.044643,780.73215)
\curveto(354.70149725,780.73215)(356.044643,779.38900425)(356.044643,777.73215)
\closepath
}
}
{
\newrgbcolor{curcolor}{0 0 0}
\pscustom[linestyle=none,fillstyle=solid,fillcolor=curcolor]
{
\newpath
\moveto(355.776786,888.178571)
\curveto(355.776786,886.52171675)(354.43364025,885.178571)(352.776786,885.178571)
\curveto(351.11993175,885.178571)(349.776786,886.52171675)(349.776786,888.178571)
\curveto(349.776786,889.83542525)(351.11993175,891.178571)(352.776786,891.178571)
\curveto(354.43364025,891.178571)(355.776786,889.83542525)(355.776786,888.178571)
\closepath
}
}
{
\newrgbcolor{curcolor}{0 0 0}
\pscustom[linestyle=none,fillstyle=solid,fillcolor=curcolor]
{
\newpath
\moveto(390.089286,888.357142)
\curveto(390.089286,886.70028775)(388.74614025,885.357142)(387.089286,885.357142)
\curveto(385.43243175,885.357142)(384.089286,886.70028775)(384.089286,888.357142)
\curveto(384.089286,890.01399625)(385.43243175,891.357142)(387.089286,891.357142)
\curveto(388.74614025,891.357142)(390.089286,890.01399625)(390.089286,888.357142)
\closepath
}
}
\rput(136,835){{\tiny $C_1$}}
\rput(370,835){{\tiny $C_2$}}

\rput(145,724){{\small $U_1$}}
\rput(381,724){{\small $U_2$}}
\rput(603,724){{\small $C_3=\emptyset$ \ and \ $U_3=\emptyset$}}

\rput(102,944){{\small $\partial A_i$}}
\rput(164,944){{\small $\partial A_j$}}

\rput(336,944){{\small $\partial A_i^{1/k}$}}
\rput(416,944){{\small $\partial A_j^{1/k}$}}

\rput(567,944){{\small $\partial A_i^{2/k}$}}
\rput(647,944){{\small $\partial A_j^{2/k}$}}

\end{pspicture}
\caption{Saturation and missing product structures after smoothing}\label{aus}\end{figure}

Proceeding in this fashion $k-2$ times more, in each step replacing $A^{t/k}_i$ by $A^{(t+1)/k}_i$, we obtain a smooth submanifold $N'_k\subset M$ which satisfies the local product structures except on an open $1/k$-neighborhood $U_k$ of
\[
C_k:=\bigcup_{1\leq i_1,\ldots,i_k\leq k}\left(\partial A_{i_1}\cap\partial A^{1/k}_{i_2}\cap\ldots\cap\partial A^{(k-1)/k}_{i_k} \right).
\]
Since the interiors of $A_{i_1}$, $A^{1/k}_{i_2},\ldots,A^{(k-1)/k}_{i_k}$ cover $M$ for all pairwise disjoint $1\leq i_1,\ldots,i_k\leq k$ and $\partial A^{s/k}_{i_j}\cap \partial A^{t/k}_{i_j}=\emptyset$ for $s\neq t$, one concludes that $C_k=\emptyset$ and thus $N:=N'_k$ is as desired (see figure \ref{aus}).

Similarly one sees that $i_*$ is monic: Let $(M,f)$ be an $n$-dimensional $\mathscr{P}$-manifold in $(X-U,A-U)$ and $i_*[(M,f)]=0$. Then there exists a zero $\mathscr{P}$-bordism $(W,g)$ for $(M,i\circ f)$ in $(X,A)$. As above we find an $(n+1)$-dimensional $\mathscr{P}$-submanifold $N\subset W$ with $g(N)\cap U=\emptyset$ and $g^{-1}(X-\mathring{A})\subset N$. It follows that $(N,g|_N)$ is a zero $\mathscr{P}$-bordism for $(M,f)$ in $(X-U,A-U)$.
\end{proof}


Now let $\mathscr{P}$ be a, possibly infinite, family of smooth closed manifolds. It is not difficult to extend our definition to that case: Let us denote the family of all finite subsets of $\mathscr{P}$ by $\mathfrak{F}$. The bordism spanned by $\mathscr{F}\in\mathfrak{F}$ is denoted by $\mathcal{F}_*(\_)$. By taking inclusions $\mathfrak{F}$ becomes a directed set and thus we can form the direct limit
\[
\mathcal{P}_*(\_):=\lim_{\mathscr{F}\in\mathfrak{F}} \mathcal{F}_*(\_).
\]
As the direct limit preserves exactness $\mathcal{P}_*(\_)$ is again a homology theory.

We turn to the computation of the coefficients groups $\mathcal{P}_*$. Let $G=SO$ or $Spin$. As usual let $MG$ denote the Thom spectrum associated to the oriented resp.~ spin bordism theory, i.e.~ $\Omega_*^G(\_)=MG_*(\_)$. From now on we fix a \emph{regular sequence} $P_1, P_2,\ldots$ of smooth closed manifolds which means that for all $i\geq 1$
\[
MG_*/\left([P_1],\ldots,[P_{i-1}]\right)\xto{\x[P_i]}MG_*/\left([P_1],\ldots,[P_{i-1}]\right)
\]
is injective, here $\left([P_1],\ldots,[P_{i-1}]\right)$ denotes the ideal generated by $[P_1],\ldots,[P_{i-1}]$.

\begin{rem} One can show cf.~ \cite[Prop.~ 2.7.1.]{cw} that in our situation any permutation of $P_1,P_2,\ldots$ is again a regular sequence.
\end{rem}

Let $\mathscr{P}=\{P_1,P_2,\ldots\}$. We have a natural transformation of homology theories $\iota_*\colon\mathcal{P}_*(\_)\to MG_*(\_)$ by forgetting the $\mathscr{P}$-structure. The following proposition is the crucial step in determining the coefficients.

\begin{prop}\label{inj} $\iota_*$ is injective on coefficients. \end{prop}
\begin{proof}
Let $\mathfrak{R}_k$ denote the family of all subsets of $\mathscr{P}$ consisting of $k$ elements. For $\mathscr{R}\in\mathfrak{R}_k$ we have the bordism spanned by $\mathscr{R}$, denoted by $\mathcal{R}_*(\_)$, and the forgetful map $\iota^{\mathcal{R}}_*\colon\mathcal{R}_*(\_)\to MG_*(\_)$. We shall prove the following statement by induction over $k$ from which Prop.~ \ref{inj} follows immediately:
\[
\text{Let }\mathscr{R}\in\mathfrak{R}_k,\text{ then }\iota^{\mathcal{R}}_*\colon\mathcal{R}_*\to MG_*\text{ is injective}.
\]

$k=1$: Let $\mathscr{R}\in\mathfrak{R}_1$. In this case a closed $\mathscr{R}$-manifold $M$ is diffeomorphic to $P_l\x B_l$ for $\{P_l\}=\mathscr{R}$. If $\iota^{\mathscr{R}}_*([M])=0$ then $[B_l]=0$ in $MG_*$ as $[P_l]$ is not a zero divisor. Hence $B_l=\partial W$ for some manifold $W$. Now the $P_l$-manifold $P_l\x W$ establishes a zero $\mathscr{R}$-bordism for $M$.

$k-1\to k$: Assume that $\iota^{\mathcal{S}}_*\colon\mathcal{S}_*\to MG_*$ is injective for all $\mathscr{S}\in\mathfrak{R}_{k-1}$. Let $\mathscr{R}:=\{P_{i_1},\ldots,P_{i_k}\}\in\mathfrak{R}_k$ and
\[
\left(M,(A_{i_j})_{1\leq j\leq k} ,(B_I,\phi_I)_{I\subset\{i_1,\ldots,i_k\}}\right)
\]
a closed $\mathscr{R}$-manifold. One has to show that $[M]=0$ in $\mathcal{R}_*$ if $\iota_*^{\mathcal{R}}[M]=0$. We shall apply, step by step, surgery to the submanifolds $\partial A_{i_j}\subset M$ in order to isolate the $P_{i_j}$-parts and to obtain a disjoint union of $P_{i_j}$-manifolds.

It follows from Def.~ \ref{fdef} that $C_{i_j}$, together with $\phi_{i_j}$, induces a smooth structure on $\partial B_{i_j}$ such that $\phi_{i_j}\colon\partial A_{i_j}\to P_{i_j}\x\partial B_{i_j}$ is a diffeomorphism for all $1\leq j\leq k$. The induction hypothesis becomes applicable by means of

\begin{lem}\label{part} Let $\mathscr{S}:=\mathscr{R}-\{P_{i_j}\}$, i.e.~ $\mathscr{S}\in\mathfrak{R}_{k-1}$. Then $\partial B_{i_j}$ inherits the structure of an $\mathscr{S}$-manifold.\end{lem}

\begin{proof} Let $J\subset\{i_1,\ldots,\hat{i}_j,\ldots,i_k\}$. There are manifolds $C_{J}\subset B_{\{i_j\}\cup J}$ 
such that
\[
\phi_{\{i_j\}\cup J}(\partial A_{i_j}\cap A_J)=P_{i_j}\x P_J\x C_{J}.
\]
Then we have an inclusion $\phi_{i_j}^{\{i_j\}\cup J}\colon P_J\x C_{J}\hto\partial B_{i_j}$ (cf.~ point four of Def.~ \ref{fdef})and observe
\[
\bigcup_{l\neq j} \phi_{i_j}^{\{i_j,i_l\}}(P_l\x C_l)=\partial B_{i_j}.
\]
Now the induced $\mathscr{S}$-structure on $\partial B_{i_j}$ is defined by setting
\[
\left(\partial B_{i_j},\left(\phi_{i_j}^{\{i_j,i_l\}}\left(P_l\x C_l\right)\right)_{l\neq j},\left(C_{J},\left(\phi^{\{i_j\}\cup J}_{i_j}\right)^{-1}\bigg\vert_{\im \phi^{\{i_j\}\cup J}_{i_j}} \right)_{J\subset\{i_1,\ldots,\hat{i}_j,\ldots,i_k\}}\right).
\]
\end{proof}

In the first surgery step we consider $\partial A_{i_1}$. As $[P_{i_1}]$ is not a zero divisor, $\partial B_{i_1}$ is zero bordant in $MG_*$. Hence, by induction hypothesis, $\partial B_{i_1}$ is zero $\mathscr{S}$-bordant with $\mathscr{S}:=\mathscr{R}-\{i_1\}$, i.e.~ there exists an $\mathscr{S}$-manifold $(N,C_i)$ with $\partial N=\partial B_{i_1}$. By abuse of notation, we shall use the indices $\{1,\ldots,k\}$ instead of $\{i_1\ldots,i_k\}$ in the sequel. Fix bicollar neighborhoods $\partial A_i\x[-1,1]$ of $\partial A_i\subset M$, say $\partial A_i\x\{-1\}\subset A_i$, such that $(\partial A_i\x[-1,1])\cap A_j$ is a bicollar neighborhood for $(\partial A_i)\cap A_j$ in $A_j$ for all $j$.

Now we attach $P_1\x N\x[-1,1]$ to $M\x[0,1]$ by identifying
\[
(x,y,1)\in(\partial A_1\x[-1,1]\x[0,1])\subset (M\x[0,1])\, \text{ with }\, (\phi_1(x),y)\in (P_1\x\partial N)\x[-1,1]
 \]
to obtain an oriented (resp.~ spin) manifold $W$. It is well-known that the corners of $W$ can be smoothened in a canonical way.

Let us equip $W$ with the structure of an $\mathscr{R}$-manifold (see figure \ref{henkel}). On $M\x[0,1]$ we simply set $A'_i:=A_i\x[0,1]$. On $P_1\x N\x[-1,1]$ we define $A''_1:=P_1\x N\x[-1,0]$ and, for $i>1$, $A''_i:=P_1\x C_i\x[-1,1]$. Note that this $\mathscr{R}$-structure on $W$ induces the given one on $M\x\{0\}$.

\begin{figure}
\psset{xunit=.6pt,yunit=.6pt,runit=.6pt}
\begin{pspicture}(150,375)(600,640)
{
\newrgbcolor{curcolor}{0 0 0}
\pscustom[linewidth=1,linecolor=curcolor]
{
\newpath
\moveto(80,420.00000262)
\lineto(640,420.00000262)
}
}
{
\newrgbcolor{curcolor}{0 0 0}
\pscustom[linewidth=1,linecolor=curcolor]
{
\newpath
\moveto(80,480.00000262)
\lineto(220,480.00000262)
}
}
{
\newrgbcolor{curcolor}{0 0 0}
\pscustom[linewidth=1,linecolor=curcolor]
{
\newpath
\moveto(500,480.00000262)
\lineto(640,480.00000262)
}
}
{
\newrgbcolor{curcolor}{0 0 0}
\pscustom[linewidth=1,linecolor=curcolor]
{
\newpath
\moveto(279.99999,480.00000262)
\lineto(440,480.00000262)
}
}
{
\newrgbcolor{curcolor}{0 0 0}
\pscustom[linewidth=0.99999989,linecolor=curcolor]
{
\newpath
\moveto(250.00005104,479.98498697)
\curveto(249.99168233,540.22500918)(299.23358863,589.06592302)(359.9849271,589.0742213)
\curveto(420.73626557,589.08251958)(469.99174024,540.25505991)(470.00010896,480.0150377)
\lineto(470.00010616,479.9711848)
}
}
{
\newrgbcolor{curcolor}{0 0 0}
\pscustom[linewidth=0.99999986,linecolor=curcolor]
{
\newpath
\moveto(279.99991076,480.01006419)
\curveto(279.99382442,523.82099209)(315.80612178,559.34165886)(359.98891571,559.34769397)
\curveto(404.17170964,559.35372909)(439.9938749,523.84284717)(439.99996124,480.03191926)
\lineto(439.99995921,480.00002625)
}
}
{
\newrgbcolor{curcolor}{0 0 0}
\pscustom[linewidth=0.99999981,linecolor=curcolor]
{
\newpath
\moveto(219.99990133,479.98274349)
\curveto(219.98925024,556.65183958)(282.66075025,618.81298392)(359.9806145,618.82354536)
\curveto(437.30047874,618.8341068)(499.98924758,556.69008595)(499.99989867,480.02098986)
\lineto(499.99989511,479.96517711)
}
}
{
\newrgbcolor{curcolor}{0 0 0}
\pscustom[linewidth=1,linecolor=curcolor]
{
\newpath
\moveto(432,599.00000262)
\lineto(395.66274,551.25478262)
}
}
{
\newrgbcolor{curcolor}{0 0 0}
\pscustom[linewidth=1,linecolor=curcolor]
{
\newpath
\moveto(325,480.00000262)
\lineto(325,420.00000262)
}
}
{
\newrgbcolor{curcolor}{0 0 0}
\pscustom[linewidth=0.99999999,linecolor=curcolor]
{
\newpath
\moveto(465.00000395,415.00000831)
\curveto(439.80000345,392.50004281)(372.6000021,422.49999431)(360.00000185,400.00003631)
}
}
{
\newrgbcolor{curcolor}{0 0 0}
\pscustom[linewidth=0.99999999,linecolor=curcolor]
{
\newpath
\moveto(254.99999975,415.00000831)
\curveto(280.20000025,392.50004281)(347.4000016,422.49999431)(360.00000185,400.00003631)
}
}
{
\newrgbcolor{curcolor}{0 0 0}
\pscustom[linewidth=1,linecolor=curcolor]
{
\newpath
\moveto(315,530.00002262)
\lineto(345,520.00000262)
}
}
{
\newrgbcolor{curcolor}{0 0 0}
\pscustom[linestyle=none,fillstyle=solid,fillcolor=curcolor]
{
\newpath
\moveto(310.62087877,531.45973261)
\lineto(318.45067922,533.06616497)
\lineto(315.92085254,525.4767001)
\lineto(310.62087877,531.45973261)
\closepath
}
}
{
\newrgbcolor{curcolor}{0 0 0}
\pscustom[linewidth=1,linecolor=curcolor]
{
\newpath
\moveto(310.62087877,531.45973261)
\lineto(318.45067922,533.06616497)
\lineto(315.92085254,525.4767001)
\lineto(310.62087877,531.45973261)
\closepath
}
}
{
\newrgbcolor{curcolor}{0 0 0}
\pscustom[linewidth=1.00000005,linecolor=curcolor]
{
\newpath
\moveto(319.99999925,415.00000831)
\curveto(291.79999937,392.50004281)(216.59999969,422.49999431)(202.49999975,400.00003631)
}
}
{
\newrgbcolor{curcolor}{0 0 0}
\pscustom[linewidth=1.00000005,linecolor=curcolor]
{
\newpath
\moveto(85.00000025,415.00000831)
\curveto(113.20000013,392.50004281)(188.39999981,422.49999431)(202.49999975,400.00003631)
}
}
{
\newrgbcolor{curcolor}{0 0 0}
\pscustom[linewidth=1,linecolor=curcolor]
{
\newpath
\moveto(228,557.00000262)
\lineto(198,567.00002262)
}
}
{
\newrgbcolor{curcolor}{0 0 0}
\pscustom[linestyle=none,fillstyle=solid,fillcolor=curcolor]
{
\newpath
\moveto(232.37912123,555.54029262)
\lineto(224.54932078,553.93386027)
\lineto(227.07914746,561.52332513)
\lineto(232.37912123,555.54029262)
\closepath
}
}
{
\newrgbcolor{curcolor}{0 0 0}
\pscustom[linewidth=1,linecolor=curcolor]
{
\newpath
\moveto(232.37912123,555.54029262)
\lineto(224.54932078,553.93386027)
\lineto(227.07914746,561.52332513)
\lineto(232.37912123,555.54029262)
\closepath
}
}
{
\newrgbcolor{curcolor}{0 0 0}
\pscustom[linewidth=0.99999994,linecolor=curcolor]
{
\newpath
\moveto(454.24239,559.24239262)
\lineto(494.24239,594.24237262)
}
}
{
\newrgbcolor{curcolor}{0 0 0}
\pscustom[linestyle=none,fillstyle=solid,fillcolor=curcolor]
{
\newpath
\moveto(450.76849532,556.20273651)
\lineto(453.34230961,563.769894)
\lineto(458.61034445,557.74927931)
\lineto(450.76849532,556.20273651)
\closepath
}
}
{
\newrgbcolor{curcolor}{0 0 0}
\pscustom[linewidth=0.99999994,linecolor=curcolor]
{
\newpath
\moveto(450.76849532,556.20273651)
\lineto(453.34230961,563.769894)
\lineto(458.61034445,557.74927931)
\lineto(450.76849532,556.20273651)
\closepath
}
}
{
\newrgbcolor{curcolor}{0 0 0}
\pscustom[linewidth=1,linecolor=curcolor]
{
\newpath
\moveto(250,480.00000262)
\lineto(250,420.00000262)
}
}
{
\newrgbcolor{curcolor}{0 0 0}
\pscustom[linewidth=1,linecolor=curcolor]
{
\newpath
\moveto(470,480.00000262)
\lineto(470,420.00000262)
}
}
{
\newrgbcolor{curcolor}{0 0 0}
\pscustom[linewidth=0.99999994,linecolor=curcolor]
{
\newpath
\moveto(644.99997778,475.00000332)
\curveto(660.00002118,469.00000365)(639.99996498,453.00000455)(655.00000338,450.00000472)
}
}
{
\newrgbcolor{curcolor}{0 0 0}
\pscustom[linewidth=0.99999994,linecolor=curcolor]
{
\newpath
\moveto(644.99997778,425.00000612)
\curveto(660.00002118,431.00000578)(639.99996498,447.00000489)(655.00000338,450.00000472)
}
}
{
\newrgbcolor{curcolor}{0.7019608 0.7019608 0.7019608}
\pscustom[linestyle=none,fillstyle=solid,fillcolor=curcolor]
{
\newpath
\moveto(325.50668,449.98475262)
\lineto(325.50668,479.27918262)
\lineto(382.72461,479.27918262)
\curveto(434.00792,479.27918262)(439.98304,479.31968262)(440.33301,479.66965262)
\curveto(440.66337,480.00002262)(440.69721,480.55395262)(440.55279,483.26832262)
\curveto(440.19494,489.99446262)(439.20327,496.11550262)(437.52908,501.93210262)
\curveto(431.69002,522.21857262)(418.26625,539.11640262)(399.60914,549.66567262)
\lineto(396.57728,551.37997262)
\lineto(405.17302,562.66390262)
\curveto(409.90068,568.87007262)(413.8398,574.01686262)(413.92663,574.10123262)
\curveto(414.27369,574.43846262)(421.4213,569.85619262)(426.2694,566.18840262)
\curveto(448.02983,549.72570262)(462.97389,525.45911262)(467.66001,498.97715262)
\curveto(469.29606,489.73152262)(469.44815,485.60883262)(469.45101,450.42670262)
\lineto(469.45301,420.69033262)
\lineto(397.47964,420.69033262)
\lineto(325.50627,420.69033262)
\lineto(325.50627,449.98475262)
\closepath
}
}
{
\newrgbcolor{curcolor}{0.80000001 0.80000001 0.80000001}
\pscustom[linestyle=none,fillstyle=solid,fillcolor=curcolor]
{
\newpath
\moveto(250.84257,455.60841262)
\curveto(250.96436,489.17981262)(250.99834,490.69724262)(251.72807,495.16198262)
\curveto(254.29497,510.86711262)(259.05242,523.94823262)(266.7131,536.36492262)
\curveto(278.38939,555.29025262)(296.08018,570.62255262)(316.18602,579.24223262)
\curveto(330.64354,585.44039262)(344.31938,588.24947262)(360.03729,588.24947262)
\curveto(374.51921,588.24947262)(387.53068,585.80589262)(400.5076,580.64906262)
\curveto(404.41715,579.09546262)(410.92197,576.02741262)(411.96635,575.24445262)
\curveto(412.60706,574.76411262)(412.16642,574.10709262)(404.16202,563.60793262)
\curveto(399.49617,557.48784262)(395.54953,552.40069262)(395.39173,552.30317262)
\curveto(395.23393,552.20567262)(394.00481,552.62081262)(392.66034,553.22577262)
\curveto(383.36961,557.40625262)(370.9374,559.98341262)(360.06148,559.98341262)
\curveto(331.37278,559.98341262)(305.03324,545.03762262)(290.38501,520.44703262)
\curveto(286.01835,513.11653262)(282.68424,504.44427262)(280.96331,495.94052262)
\curveto(279.95037,490.93520262)(279.06789,481.76281262)(279.45201,480.23232262)
\lineto(279.69376,479.26912262)
\lineto(302.00838,479.26912262)
\lineto(324.32301,479.26912262)
\lineto(324.32301,449.98341262)
\lineto(324.32301,420.69770262)
\lineto(287.51946,420.69770262)
\lineto(250.71591,420.69770262)
\lineto(250.84257,455.60841262)
\lineto(250.84257,455.60841262)
\closepath
}
}
{
\newrgbcolor{curcolor}{0.90196079 0.90196079 0.90196079}
\pscustom[linestyle=none,fillstyle=solid,fillcolor=curcolor]
{
\newpath
\moveto(80.719597,449.98341262)
\lineto(80.719597,479.26912262)
\lineto(150.27317,479.27412262)
\curveto(188.52763,479.27712262)(220.00357,479.39766262)(220.21969,479.54198262)
\curveto(220.46022,479.70259262)(220.69867,481.95161262)(220.83445,485.34008262)
\curveto(222.45806,525.86098262)(241.57228,563.05447262)(273.81468,588.43190262)
\curveto(289.19033,600.53380262)(307.71035,609.48518262)(327.14817,614.20986262)
\curveto(354.3366,620.81844262)(383.86367,618.91972262)(410.16124,608.87176262)
\curveto(415.20878,606.94316262)(429.63666,599.96460262)(430.34769,599.10787262)
\curveto(430.70941,598.67202262)(429.33672,596.67658262)(422.29894,587.40767262)
\curveto(417.63228,581.26157262)(413.77694,576.18222262)(413.73152,576.12023262)
\curveto(413.68612,576.05823262)(412.78894,576.49149262)(411.73783,577.08300262)
\curveto(405.80851,580.41979262)(392.8759,585.16141262)(384.83122,586.94806262)
\curveto(375.28497,589.06820262)(371.79459,589.41504262)(360.00531,589.41504262)
\curveto(350.65651,589.41504262)(348.42838,589.30758262)(344.28654,588.65696262)
\curveto(326.01807,585.78726262)(309.84698,579.38429262)(295.89817,569.49751262)
\curveto(285.57932,562.18361262)(275.56361,552.04456262)(268.50665,541.76865262)
\curveto(259.93128,529.28174262)(254.08715,515.23399262)(251.21069,500.19369262)
\curveto(249.55428,491.53277262)(249.53831,491.13674262)(249.39167,455.07223262)
\lineto(249.2519,420.69723262)
\lineto(164.98575,420.69723262)
\lineto(80.719597,420.69723262)
\lineto(80.719597,449.98294262)
\closepath
}
}
{
\newrgbcolor{curcolor}{0 0 0}
\pscustom[linewidth=1,linecolor=curcolor,linestyle=dashed,dash=2 4]
{
\newpath
\moveto(220,480.00000262)
\lineto(280,480.00000262)
}
}
{
\newrgbcolor{curcolor}{0 0 0}
\pscustom[linewidth=1,linecolor=curcolor,linestyle=dashed,dash=2 4]
{
\newpath
\moveto(440,480.00000262)
\lineto(500,480.00000262)
}
}
{
\newrgbcolor{curcolor}{0 0 0}
\pscustom[linestyle=none,fillstyle=solid,fillcolor=curcolor]
{
\newpath
\moveto(474,419.99999652)
\curveto(474,417.79085752)(472.209139,415.99999652)(470,415.99999652)
\curveto(467.790861,415.99999652)(466,417.79085752)(466,419.99999652)
\curveto(466,422.20913552)(467.790861,423.99999652)(470,423.99999652)
\curveto(472.209139,423.99999652)(474,422.20913552)(474,419.99999652)
\closepath
}
}
{
\newrgbcolor{curcolor}{0 0 0}
\pscustom[linestyle=none,fillstyle=solid,fillcolor=curcolor]
{
\newpath
\moveto(329,419.99999652)
\curveto(329,417.79085752)(327.209139,415.99999652)(325,415.99999652)
\curveto(322.790861,415.99999652)(321,417.79085752)(321,419.99999652)
\curveto(321,422.20913552)(322.790861,423.99999652)(325,423.99999652)
\curveto(327.209139,423.99999652)(329,422.20913552)(329,419.99999652)
\closepath
}
}
{
\newrgbcolor{curcolor}{0 0 0}
\pscustom[linestyle=none,fillstyle=solid,fillcolor=curcolor]
{
\newpath
\moveto(254,419.99999652)
\curveto(254,417.79085752)(252.209139,415.99999652)(250,415.99999652)
\curveto(247.790861,415.99999652)(246,417.79085752)(246,419.99999652)
\curveto(246,422.20913552)(247.790861,423.99999652)(250,423.99999652)
\curveto(252.209139,423.99999652)(254,422.20913552)(254,419.99999652)
\closepath
}
}
{
\newrgbcolor{curcolor}{0 0 0}
\pscustom[linestyle=none,fillstyle=solid,fillcolor=curcolor]
{
\newpath
\moveto(254,479.99999952)
\curveto(254,477.79086052)(252.209139,475.99999952)(250,475.99999952)
\curveto(247.790861,475.99999952)(246,477.79086052)(246,479.99999952)
\curveto(246,482.20913852)(247.790861,483.99999952)(250,483.99999952)
\curveto(252.209139,483.99999952)(254,482.20913852)(254,479.99999952)
\closepath
}
}
{
\newrgbcolor{curcolor}{0 0 0}
\pscustom[linestyle=none,fillstyle=solid,fillcolor=curcolor]
{
\newpath
\moveto(474,479.99999952)
\curveto(474,477.79086052)(472.209139,475.99999952)(470,475.99999952)
\curveto(467.790861,475.99999952)(466,477.79086052)(466,479.99999952)
\curveto(466,482.20913852)(467.790861,483.99999952)(470,483.99999952)
\curveto(472.209139,483.99999952)(474,482.20913852)(474,479.99999952)
\closepath
}
}
\rput(175,577){$M_{>1}$}
\rput(364,513){$M_1$}
\rput(360,380){$A_1$}
\rput(202,380){$A_i$}
\rput(395,448){$A'_1$}
\rput(167,448){$A'_i$}
\rput(702,448){$M\x[0,1]$}
\rput(536,610){$P_1\x N$}
\rput(286,448){\small{$A'_i\cap A'_1$}}
\end{pspicture}
\caption{$\mathscr{R}$-structure of $W$}\label{henkel}\end{figure}
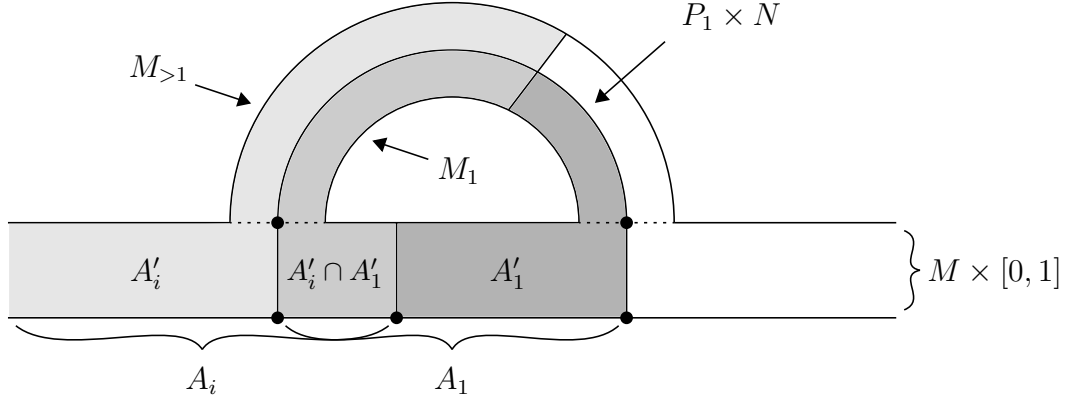

By construction $W$ is an $\mathscr{R}$-bordism between $M$ and the disjoint union of
\begin{itemize}
\item $M_1:=A_1\cup (P_1\x N)\text{ glued along }\phi_1(\partial A_1)=\partial (P_1\x N)$,
\item $M_{>1} :=(M-\mathring{A}_1)\cup (P_1\x N)\text{ glued along }\phi_1(\partial (M-\mathring{A}_1))=\partial (P_1\x N)$.
\end{itemize}
The trace of the bordism induces an $\mathscr{R}$-structure on $M_{>1}$ with an empty $P_1$-part. A priori $W$ induces on $M_1$ an $\mathscr{R}$-structure with non-empty $P_i$-parts, $i>1$. However, $M_1$ is completely covered by the $P_1$-part $P_1\x(B_1\cup N)$. The following lemma shows that we can ignore redundant subsets $A_i$, more precisely:

\begin{lem}
Let $(M,(A_i)_{1\leq i\leq k})$ be a closed $\mathscr{R}$-manifold, $1\leq j \leq k$ and $M=\cup_{i=1,i\neq j}^k A_i$. Then $(M,(A_i)_{1\leq i\leq k})$ is $\mathscr{R}$-bordant to $(M,(A'_i)_{1\leq i\leq k})$ with $A'_i= A_i$ for $i\neq j$ and $A'_j=\emptyset$.
\end{lem}
\begin{proof}
Let $A_j$ be diffeomorphic to $P_j\x B_j$. Choose a collar neighborhood $\partial B_j\x[-1,0]$, say $\partial B_j\x\{0\}=\partial B_j$, such that for the induced collar neighborhood $\partial A_j\x[-1,0]$ it is true that $(\partial A_j\x[-1,0])\cap A_i$ is a collar neighborhood for $(\partial A_j)\cap A_i$ in $A_i$ for all $i$. Let $\gc\colon[-1,0]\to[-1,0]\x[0,0,\!5]$ be a smooth injective convex curve with $\gc(t)=(t,0,\!5)$ for $t<-0,\!9$ and $\gc(t)=(0,-t)$ for $t>-0,\!1$. Now we define the desired $\mathscr{R}$-bordism $M\x[0,1]=\cup_{i=1}^k\hat{A}_i$ by (see figure \ref{red}) $\hat{A}_i:=A_i\x[0,1]$ for $i\neq j$ and
\[
\hat{A}_j:=\left\{(x,s)\in A_j\x[0,1]\, |\, s\leq\begin{cases} 0,\!5     &\text{ if }x\in A_j-(\partial A_j\x[-1,0])\\
                                             \gc(t)_2&\text{ if }x=(y,t)\in\partial A_j\x[-1,0)\\
                                             {\displaystyle\max_{\{t\, |\, \gc(t)_1=0\}}} \gc(t)_2&\text{ if }x\in\partial A_j\x\{0\} \end{cases}\right\}.
\]

\end{proof}

\begin{figure}
\psset{xunit=.5pt,yunit=.5pt,runit=.5pt}
\begin{pspicture}(150,360)(600,675)
{
\newrgbcolor{curcolor}{0 0 0}
\pscustom[linewidth=0.99999996,linecolor=curcolor]
{
\newpath
\moveto(458.77716819,417.20580932)
\curveto(458.77716716,406.11961479)(416.24412398,397.1324786)(363.77695732,397.13247882)
\curveto(311.40105739,397.13247903)(268.90591557,406.08954007)(268.77703779,417.15641665)
}
}
{
\newrgbcolor{curcolor}{0 0 0}
\pscustom[linewidth=0.99999996,linecolor=curcolor]
{
\newpath
\moveto(458.77695819,631.04824932)
\curveto(458.77695716,619.96205479)(416.24391398,610.9749186)(363.77674732,610.97491882)
\curveto(311.40084739,610.97491903)(268.90570557,619.93198007)(268.77682779,630.99885665)
}
}
{
\newrgbcolor{curcolor}{0 0 0}
\pscustom[linewidth=0.99999996,linecolor=curcolor]
{
\newpath
\moveto(268.77737654,630.96401972)
\curveto(268.5572183,642.05011666)(310.91141268,651.07488496)(363.37811743,651.1214039)
\curveto(415.84482219,651.16792284)(458.55596359,642.21857665)(458.77612183,631.13247971)
\curveto(458.77707134,631.0846671)(458.77721237,631.03685393)(458.77654492,630.98904111)
}
}
{
\newrgbcolor{curcolor}{0 0 0}
\pscustom[linewidth=1,linecolor=curcolor]
{
\newpath
\moveto(458.77675,632.12158262)
\lineto(458.77675,417.12158262)
}
}
{
\newrgbcolor{curcolor}{0 0 0}
\pscustom[linewidth=1,linecolor=curcolor]
{
\newpath
\moveto(268.77675,632.12158262)
\lineto(268.77675,417.12158262)
}
}
{
\newrgbcolor{curcolor}{0 0 0}
\pscustom[linewidth=1,linecolor=curcolor]
{
\newpath
\moveto(303.77675,401.82158)
\curveto(303.77675,501.82158)(303.77675,501.82158)(338.77675,501.82158)
\lineto(388.77675,501.82158)
}
}
{
\newrgbcolor{curcolor}{0 0 0}
\pscustom[linewidth=1,linecolor=curcolor]
{
\newpath
\moveto(423.77675,401.82158)
\curveto(423.77675,501.82158)(423.77675,501.82158)(388.77675,501.82158)
}
}
{
\newrgbcolor{curcolor}{0.80000001 0.80000001 0.80000001}
\pscustom[linestyle=none,fillstyle=solid,fillcolor=curcolor]
{
\newpath
\moveto(349.34999,397.95433262)
\curveto(334.0154,398.49176262)(318.5699,399.87093262)(306.93928,401.74129262)
\lineto(304.43928,402.14333262)
\lineto(304.46268,423.56988262)
\curveto(304.50708,464.14860262)(305.42254,478.92808262)(308.47238,488.30001262)
\curveto(311.30432,497.00237262)(315.85851,500.18935262)(326.58212,500.97301262)
\curveto(329.99908,501.22272262)(397.45038,501.21939262)(400.93843,500.96901262)
\curveto(411.6482,500.20126262)(416.1463,497.09484262)(418.99368,488.49990262)
\curveto(422.15381,478.96090262)(423.09567,463.50841262)(423.10383,421.06753262)
\lineto(423.10783,402.13896262)
\lineto(421.18444,401.82828262)
\curveto(408.94874,399.85187262)(393.75249,398.50013262)(377.53179,397.94528262)
\curveto(371.60567,397.74257262)(355.23859,397.74763262)(349.35033,397.95428262)
\closepath
}
}
{
\newrgbcolor{curcolor}{0 0 0}
\pscustom[linewidth=0.17857143,linecolor=curcolor]
{
\newpath
\moveto(349.34999,397.95433262)
\curveto(334.0154,398.49176262)(318.5699,399.87093262)(306.93928,401.74129262)
\lineto(304.43928,402.14333262)
\lineto(304.46268,423.56988262)
\curveto(304.50708,464.14860262)(305.42254,478.92808262)(308.47238,488.30001262)
\curveto(311.30432,497.00237262)(315.85851,500.18935262)(326.58212,500.97301262)
\curveto(329.99908,501.22272262)(397.45038,501.21939262)(400.93843,500.96901262)
\curveto(411.6482,500.20126262)(416.1463,497.09484262)(418.99368,488.49990262)
\curveto(422.15381,478.96090262)(423.09567,463.50841262)(423.10383,421.06753262)
\lineto(423.10783,402.13896262)
\lineto(421.18444,401.82828262)
\curveto(408.94874,399.85187262)(393.75249,398.50013262)(377.53179,397.94528262)
\curveto(371.60567,397.74257262)(355.23859,397.74763262)(349.35033,397.95428262)
\closepath
}
}
{
\newrgbcolor{curcolor}{0 0 0}
\pscustom[linewidth=0.99999996,linecolor=curcolor,linestyle=dashed,dash=2.99451197 5.98902393]
{
\newpath
\moveto(268.77758654,417.12157972)
\curveto(268.5574283,428.20767666)(310.91162268,437.23244496)(363.37832743,437.2789639)
\curveto(415.84503219,437.32548284)(458.55617359,428.37613665)(458.77633183,417.29003971)
\curveto(458.77728134,417.2422271)(458.77742237,417.19441393)(458.77675492,417.14660111)
}
}
{
\newrgbcolor{curcolor}{0 0 0}
\pscustom[linewidth=0.66234878,linecolor=curcolor]
{
\newpath
\moveto(419.99999515,396.28374293)
\curveto(406.49999522,377.85812198)(370.49999541,402.42561453)(363.74999545,383.99999973)
}
}
{
\newrgbcolor{curcolor}{0 0 0}
\pscustom[linewidth=0.66234878,linecolor=curcolor]
{
\newpath
\moveto(307.49999575,396.28374293)
\curveto(320.99999568,377.85812198)(356.99999549,402.42561453)(363.74999545,383.99999973)
}
}
{
\newrgbcolor{curcolor}{0 0 0}
\pscustom[linewidth=1,linecolor=curcolor]
{
\newpath
\moveto(495,417.00000262)
\lineto(469.616,417.00000262)
}
}
{
\newrgbcolor{curcolor}{0 0 0}
\pscustom[linestyle=none,fillstyle=solid,fillcolor=curcolor]
{
\newpath
\moveto(465,417.00000262)
\lineto(471.92,421.00000262)
\lineto(471.92,413.00000262)
\lineto(465,417.00000262)
\closepath
}
}
{
\newrgbcolor{curcolor}{0 0 0}
\pscustom[linewidth=1,linecolor=curcolor]
{
\newpath
\moveto(465,417.00000262)
\lineto(471.92,421.00000262)
\lineto(471.92,413.00000262)
\lineto(465,417.00000262)
\closepath
}
}
{
\newrgbcolor{curcolor}{0 0 0}
\pscustom[linewidth=1,linecolor=curcolor]
{
\newpath
\moveto(495,520.00000262)
\lineto(373.78165,465.54191262)
}
}
{
\newrgbcolor{curcolor}{0 0 0}
\pscustom[linestyle=none,fillstyle=solid,fillcolor=curcolor]
{
\newpath
\moveto(369.57104907,463.65027423)
\lineto(374.24410039,470.13479359)
\lineto(377.52250314,462.83739163)
\lineto(369.57104907,463.65027423)
\closepath
}
}
{
\newrgbcolor{curcolor}{0 0 0}
\pscustom[linewidth=1,linecolor=curcolor]
{
\newpath
\moveto(369.57104907,463.65027423)
\lineto(374.24410039,470.13479359)
\lineto(377.52250314,462.83739163)
\lineto(369.57104907,463.65027423)
\closepath
}
}
{
\newrgbcolor{curcolor}{0 0 0}
\pscustom[linewidth=0.97098178,linecolor=curcolor]
{
\newpath
\moveto(495.3985,630.75053262)
\lineto(469.98992,630.75053262)
}
}
{
\newrgbcolor{curcolor}{0 0 0}
\pscustom[linestyle=none,fillstyle=solid,fillcolor=curcolor]
{
\newpath
\moveto(465.50786812,630.75053262)
\lineto(472.22706201,634.63445972)
\lineto(472.22706201,626.86660551)
\lineto(465.50786812,630.75053262)
\closepath
}
}
{
\newrgbcolor{curcolor}{0 0 0}
\pscustom[linewidth=0.97098178,linecolor=curcolor]
{
\newpath
\moveto(465.50786812,630.75053262)
\lineto(472.22706201,634.63445972)
\lineto(472.22706201,626.86660551)
\lineto(465.50786812,630.75053262)
\closepath
}
}
{
\newrgbcolor{curcolor}{0 0 0}
\pscustom[linestyle=none,fillstyle=solid,fillcolor=curcolor]
{
\newpath
\moveto(307.8,401.79999652)
\curveto(307.8,399.59085752)(306.009139,397.79999652)(303.8,397.79999652)
\curveto(301.590861,397.79999652)(299.8,399.59085752)(299.8,401.79999652)
\curveto(299.8,404.00913552)(301.590861,405.79999652)(303.8,405.79999652)
\curveto(306.009139,405.79999652)(307.8,404.00913552)(307.8,401.79999652)
\closepath
}
}
{
\newrgbcolor{curcolor}{0 0 0}
\pscustom[linestyle=none,fillstyle=solid,fillcolor=curcolor]
{
\newpath
\moveto(427.7,401.89999652)
\curveto(427.7,399.69085752)(425.909139,397.89999652)(423.7,397.89999652)
\curveto(421.490861,397.89999652)(419.7,399.69085752)(419.7,401.89999652)
\curveto(419.7,404.10913552)(421.490861,405.89999652)(423.7,405.89999652)
\curveto(425.909139,405.89999652)(427.7,404.10913552)(427.7,401.89999652)
\closepath
}
}
\rput(608,629){$M$ without $P_j$-part}
\rput(520,525){$\hat{A}_j$}
\rput(520,415){$M$}
\rput(366,366){$A_j$}
\end{pspicture}
\caption{Redundant $A_j$-part}\label{red}\end{figure}
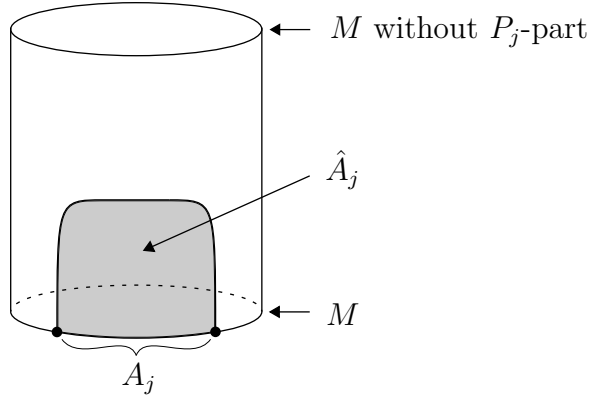
Applying this statement $(k-1)$-times to $M_1$ it follows that $M_1$ becomes a $P_1$-manifold. As noted above any permutation of $P_1,P_2,\ldots,P_k$ is again a regular sequence. Therefore we may repeat the above surgery procedure applied to the $\mathscr{R}$-manifold $M_{>1}$. This yields an $\mathscr{R}$-bordism between $M_{>1}$ on the one hand and a $P_2$-manifold $M_2$ resp.~ an $\mathscr{R}$-manifold $M_{>2}$ with empty $P_1$- and $P_2$-parts on the other. In this fashion we obtain an $\mathscr{R}$-bordism between $M$ and a disjoint union
\begin{equation}\label{dis}
(P_1\x Q_1)\,\dot{\cup}\ldots\dot{\cup}\,(P_k\x Q_k),
\end{equation}
where each $P_i \x Q_i$ is as $\mathscr{R}$-manifold a $P_i$-manifold.

To complete the induction step we have to show that if $\iota_*^{\mathcal{R}}[M]=0$ in $MG_*$ then $[M]=0$ in $\mathcal{R}_*$. Note again that any permutation of $P_1,\ldots,P_k$ is regular. In $MG_*$ we observe the following: Since $M$ is zero bordant it follows from \ref{dis} that for all $1\leq j\leq k$
\[
\sum_i[P_i\x Q_i]=0 \mod ([P_1],\ldots,[\widehat{P_j}],\ldots,[P_k]),
\]
hence $[Q_j]\in ([P_1],\ldots,[\widehat{P_j}],\ldots,[P_k])$. Then we conclude that for all $1\leq s,t\leq k$ there exists a closed manifold $Q_{st}$ such that $M$ is bordant to $\sum_{s,t}P_{\{s,t\}}\x Q_{st}$. Furthermore, for all $1\leq j,l\leq k$ we look at
\[
\sum_{s,t}[P_{\{s,t\}}\x Q_{st}]=0 \mod ([P_1],\ldots,[\widehat{P_j}],\ldots,[\widehat{P_l}],\ldots,[P_k])
\]
and conclude $[Q_{st}]\in ([P_1],\ldots,[\widehat{P_j}],\ldots,[\widehat{P_l}],\ldots,[P_k])$ for $\{s,t\}=\{j,l\}$. Proceeding in this fashion we find a closed manifold $Q$ such that $M$ is bordant to $P_1\x\ldots\x P_k\x Q$.

Now we turn to $\mathcal{R}_*$. If we remember the $P_i$-factors in \ref{dis} it follows from the above observation that there are closed manifolds $R_i$ such that $M$ is $\mathscr{R}$-bordant to
\begin{equation}\label{spec}
(\overline{P_1}\x\ldots\x P_k\x R_1) \,\dot{\cup}\ldots\dot{\cup}\,(P_1\x\ldots\x\overline{P_k}\x R_k),
\end{equation}
where $\overline{P_1}\x\ldots\x P_k\x R_1$ denotes the \emph{$P_1$-manifold} $P_1\x\ldots\x P_k\x R_1$ etc. It is not difficult to see that the specification of the $P_i$'s in \ref{spec} is immaterial. In fact, $\overline{P_i}\x P_j$ is $\mathscr{R}$-bordant to $P_i\x\overline{P_j}$ by means of the $\mathscr{R}$-bordism $[0,1]\x P_i\x P_j$ with $\left[0,\frac{2}{3}\right]\x P_i\x P_j$ as the $P_i$-part and $\left[\frac{1}{3},1\right]\x P_i\x P_j$ as the $P_j$-part.

We conclude that $M$ is $\mathscr{R}$-bordant to, say, $\overline{P_1}\x\ldots\x P_k\x Q$. Now, since $[P_1\x\ldots\x P_k]$ is not a zero divisor it follows that $Q$ is zero bordant, i.e.~ $Q=\partial R$ for some $R$. The $P_1$-manifold $\overline{P_1}\x\ldots\x P_k\x R$ is the required zero $\mathscr{R}$-bordism for $M$. This finishes the induction step.
\end{proof}

\begin{rem} Our construction of $\mathcal{P}_*(\_)$ and Prop. \ref{inj} immediately extend to smooth manifolds with other reductions of their structure groups.\end{rem}

Now we show how one obtains the desired description of $\widehat{MSO}$ and $\widehat{MSpin}$. We consider all spectra and groups after inverting $2$. Let us start with the oriented case. By a classical result of Milnor \cite{milpol} there are closed oriented manifolds $Q_1,Q_2,\ldots$, $\dim Q_i=4i$, such that
\begin{equation}\label{msop}
MSO_*\cong\mathbb{Z}{\textstyle \left[\frac{1}{2}\right]}[[Q_1],[Q_2],\ldots]
\end{equation}
(cf.~ also \cite[p.~ 180]{stong}). For example, we can take complex projective spaces and hypersurfaces of degree $(1,1)$ in $\mathbb{C}P^n\x\mathbb{C}P^m$ as generators. One concludes that the kernel of $u_*\colon MSO_*\to H\mathbb{Z}_*$ coincides with the ideal $([Q_1],[Q_2],\ldots)$. We denote the spectrum associated to the bordism spanned by $\{Q_1,Q_2,\ldots\}$ by $\mathcal{P}^u$. In view of \ref{msop} the sequence $Q_1,Q_2,\ldots$ is regular.

\begin{prop} There is a canonical homotopy equivalence $\mathcal{P}^u\simeq\widehat{MSO}$.\end{prop}

\begin{proof}
The natural transformation of homology theories $\iota_*\colon\mathcal{P}^u_*(\_)\to MSO_*(\_)$ corresponds to a spectrum map $\iota\colon\mathcal{P}^u\to MSO$. Consider the lifting problem
\begin{equation*}
\xymatrix{
                                                                & \widehat{MSO} \ar[d]_-i         &\\
\mathcal{P}^u\ar[r]^-{\iota}\ar@{-->}[ru]  & MSO\ar[r]^-u                          & H\Z.
}
\end{equation*}
Since $\widetilde{H}^0(\mathcal{P}^u;\Z\left[\frac{1}{2}\right])=0$ there exists a unique (up to homotopy) solution $h\colon\mathcal{P}^u\to\widehat{MSO}$. On the one hand $i_*\colon\widehat{MSO}_n\to MSO_n$ is an isomorphism for all $n>0$ and $\widehat{MSO}_0=0$. On the other hand $\iota_*\colon\mathcal{P}^u_n\to MSO_n$ is obviously surjective for all $n>0$ and, following Lemma \ref{inj}, $\iota_*$ is injective. In addition, $\mathcal{P}^u_0=0$. It follows that $h_*\colon\mathcal{P}^u_n\to\widehat{MSO}_n$ is an isomorphism for all $n\geq 0$. All spectra involved are $CW$-spectra and thus, according to Whitehead's theorem, $h$ is a homotopy equivalence.
\end{proof}

Now we turn to the spin case. One proceeds like in the oriented case, however, solving the corresponding lifting problem requires some more work. There are closed spin manifolds $R_1,R_2,\ldots$, $\dim R_i=4i$, such that
\[
MSpin_*\cong\mathbb{Z}{\textstyle \left[\frac{1}{2}\right]}[[R_1],[R_2],\ldots].
\]
As one can show \cite[Sec.~ 4]{hpeh}, $R_1,R_2,\ldots$ can be chosen such that $R_i$ is the total space of a $\mathbb{H}P^2$-bundle for $i\geq 2$ and the kernel of $a_*\colon MSpin_*\to ko_*$ coincides with the ideal $([R_2],[R_3],\ldots)$. Similar as above we denote the spectrum associated to the bordism spanned by the \emph{regular} sequence $R_2,R_3,\ldots$ by $\mathcal{P}^a$.

\begin{prop}\label{spinprop}  There is a canonical homotopy equivalence $\mathcal{P}^a\simeq\widehat{MSpin}$.\end{prop}

The lifting problem
\begin{equation*}
\xymatrix{
                                        & \widehat{MSpin} \ar[d]_-i         &\\
\mathcal{P}^a\ar[r]^-{\iota}\ar@{-->}[ru]  & MSpin\ar[r]^-a      & ko
}
\end{equation*}
is equivalent to the extension problem
\begin{equation*}
\xymatrix{
\mathcal{P}^a \ar[r]^-{\iota}   &   MSpin\ar[d]^-a\ar[r]^-p    &    \mathcal{C}  \\
                                             &    ko\ar@{<--}[ru]                &
}
\end{equation*}
where $\mathcal{C}$ denotes the homotopy cofiber of $\iota$.
\begin{prop}[Equiv.~ to Prop.~ \ref{spinprop}] There is a canonical homotopy equivalence $\mathcal{C}\simeq ko$.\end{prop}

\begin{rem} One can show that the homology theory associated to $\mathcal{C}$ admits a description by means of smooth spin manifolds which carry a $\mathscr{P}^a$-structure on their \emph{boundary}. Hence we recover a geometric description of $ko_*(\_)\left[\frac{1}{2}\right]$.\end{rem}

\begin{proof}
The idea is that we find an isomorphism on the level of coefficients and then apply the Conner-Floyd theorem to obtain a spectrum map. A similar argument can be found in \cite[p.~ 597]{landweber}. In order to proceed in this way one has to turn to periodic theories first.

Analogous to the oriented case the image of $\mathcal{P}^a$ under $\iota_*$ coincides with the ideal $([R_2],[R_3],\ldots)$. Hence we have $\mathcal{C}_*=\Z \left[\frac{1}{2}\right][p(R_1)]$. Consider now the $MSpin$-module spectrum $\mathcal{C}[p(R_1)^{-1}]$. As module over $MSpin_*$ one concludes that $\mathcal{C}[p(R_1)^{-1}]_*$ is isomorphic to $\Z \left[\frac{1}{2}\right][p(R_1),p(R_1)^{-1}]$.

As a generator $R_1$ of $MSpin_4$ one can use a K3 surface of signature 16. Let $\omega_4\in ko_4$ denote the image of $[R_1]\in MSpin_4$ under the orientation map $a_*\colon MSpin_*\to ko_*$. By Bott periodicity one has $KO_*=\Z \left[\frac{1}{2}\right][\omega_4,\omega_4^{-1}]$ and hence $ko_*=\Z \left[\frac{1}{2}\right][\omega_4]$. It follows that there is a unique $MSpin_*$-module isomorphism $\phi\colon KO_*\to \mathcal{C}[p(R_1)^{-1}]_*$.

Now let $MSp$ denote the symplectic Thom spectrum. The real Connor-Floyd theorem \cite{cf} states that
\[
\mu\colon MSp_*(X)\otimes_{MSp_*}KO_* \to KO_*(X),
\]
induced by the $MSp$-module structure of $KO$, is an isomorphism. After inverting $2$ one can show that there is a natural equivalence $MSpin\simeq MSp$. Consider now
\begin{equation}\label{cslp}
\xymatrix{
MSpin_*(X)\ot_{MSpin_*}KO_*  \ar[d]^-{\id\ot\phi}  \ar[r]^-{\mu}               &   KO_*(X)  \\
MSpin_*(X)\ot_{MSpin_*}\mathcal{C}[p(R_1)^{-1}]_*  \ar[r]^-{\eta}  &  \mathcal{C}[p(R_1)^{-1}]_*(X),
}
\end{equation}
where $\eta$ is induced by the $MSpin$-module structure of $\mathcal{C}[p(R_1)^{-1}]$. It follows that the natural transformation of homology theories
\[
\eta\circ(\id\ot\phi)\circ\mu^{-1}\colon KO_*(\_)\to\mathcal{C}[p(R_1)^{-1}]_*(\_)
\]
is an equivalence. Hence there exists a homotopy equivalence $KO\simeq\mathcal{C}[p(R_1)^{-1}]$. The periodization map $ko\to KO$ induces isomorphisms on non-negative homotopy groups and $\mathcal{C}$ is a connective spectrum. One concludes that the lifting problem
\begin{equation*}
\xymatrix{
                                                 &                                                        & ko\ar[d]   \\
\mathcal{C}\ar[r]\ar@{-->}[rru]  &  \mathcal{C}[p(R_1)^{-1}]\ar[r]^-{\simeq}  & KO
}
\end{equation*}
has a unique (up to homotopy) solution $h\colon\mathcal{C}\to ko$ which clearly induces isomorphisms on homotopy groups.
\end{proof}

\begin{rem} It is not possible to improve these methods in the sense that one could pass from $ko$ to $KO$ in Thm.~ \ref{sj}. In fact, in \cite{dss} it is proved that there are manifolds in the kernel of the periodization map $ko_*(X)\to KO_*(X)$ which does not admit a pscm.\end{rem}


\section{Positive Scalar Curvature}\label{pscm}
Let $\mathscr{P}=\{P_1,P_2,\ldots\}$ be a family of smooth closed manifolds. Now we shall prove our geometric result that a $\mathscr{P}$-manifold carries a pscm if all $P_i$ do. We note that neither orientability and spin structures nor regularity of the sequence $P_1,P_2,\ldots$ are needed.

\begin{thm}\label{pscmprop}  Let $\mathscr{P}=\{P_1,P_2,\ldots\}$ be a family of pscm manifolds. Then a $\mathscr{P}$-manifold, considered as a smooth manifold with additional structure, carries a pscm. \end{thm}
As a consequence we obtain a proof of Thm.~ \ref{mt} with $2$ inverted:
\begin{cor} For all spaces $X$ the kernels of $A\colon\Omega^{Spin}_*(X) \left[\frac{1}{2}\right]\to ko_*(X) \left[\frac{1}{2}\right]$ and $U\colon\Omega^{SO}_*(X) \left[\frac{1}{2}\right]\to H_*(X;\Z \left[\frac{1}{2}\right])$ are generated by manifolds which carry a pscm.\end{cor}
\begin{proof}
As mentioned above $\ker U$ is generated by projective spaces and hypersurfaces of degree $(1,1)$ in $\mathbb{C}P^n\x\mathbb{C}P^m$. In \cite{gl} it is explained why these manifolds carry a pscm: The standard Fubini-Study metric on $\CP^n$ is a pscm. Hypersurfaces of degree $(1,1)$ in $\CP^n\x\CP^m$ are projective space bundles over projective spaces, and with the induced metric the fibers being totally geodesic submanifolds. Hence the O`Neill formulas \cite[Prop.~ 9.70d]{besse} guarantee pscm on total spaces.

In \cite[Sec.~ 4]{hpeh} it is proved that $\ker A$ is generated by $\mathbb{H}P^2$-bundles with isometric action of the structure group. One concludes \cite[Thm.~ 9.59]{besse} that there exists metrics on the total spaces such that the fibers are also totally geodesic. Since $\mathbb{H}P^2$ carries a pscm the claim again follows by the O'Neill formulas.
\end{proof}

As a preliminary point, we recall that a \emph{collar metric} on a smooth manifold $M$ with boundary is a collar neighborhood $\partial M\x[0,1]$ together with a metric on $M$ which restricts to $g\x dt^2$ on $\partial M\x[0,1]$, where $g$ is some metric on $\partial M$ and $dt^2$ is the standard metric on $[0,1]$. Two pscm $g_0$ and $g_1$ on a closed manifold $M$ are called
\begin{itemize}
\item \emph{isotopic} if they lie in the same path component of the space of pscm on $M$ equipped with the $C^{\infty}$ topology,
\item \emph{concordant} if there exists a collar pscm $H$ on $M\x[0,1]$ such that $H|_{M\x\{0\}}=g_0$ and $H|_{M\x\{1\}}=g_1$.
\end{itemize}
It is well known that isotopy implies concordance.

With respect to products and scalar multiplication scalar curvature behaves as follows. If $\epsilon>0$ and $g_0$, $g_1$ are metrics on $M$, $N$ then $scal(\epsilon g_0)=\epsilon^{-1}scal(g_0)$ and $scal(g_0\x g_1)=scal(g_0)+scal(g_1)$. Hence, if $M$ and $N$ are compact and $N$ admits a pscm then $M\x N$ does.
\bigskip

The crucial step in the proof of Thm.~ \ref{pscmprop} is a simple concordance argument which can be easily demonstrated in the case of a $\mathscr{P}$-manifold $M$ consisting of two $P_i$-parts, i.e.~ $M=A_1\cup A_2$. Due to the definition of a $\mathscr{P}$-manifold there is a submanifold $Q\subset\partial B_{12}$ such that $\phi_1(\partial A_1)=P_1\x P_2\x Q$ and a submanifold $B_2'\subset B_2$ such that $\phi_2(A_2-\mathring{A}_1)=P_2\x B_2'$. In addition, $\partial A_1\hto A_2-\mathring{A}_1$ is induced by some diffeomorphism $\psi\colon P_1\x Q\to \partial B'_2$. Choose a pscm $g_1$ on $P_1$, a metric $h$ on $Q$ and extend $\psi_*(g_1\x h)$ to a collar metric $h_2$ on $B_2'$. Now take a pscm $g_2$ on $P_2$ such that $G_2:=g_2\x h_2$ is a pscm on $\phi_2(A_2-A_1)$. Next extend $g_2\x h$ to a collar metric $h_1$ on $B_1$. We find an $\epsilon>0$ small enough such that $G_1:=(\epsilon g_1)\x h_1$ is a pscm on $\phi_1(A_1)$. Note that $\phi_1^*(G_1)$ and $\phi_2^*(G_2)$ restricted to $\partial A_1$ are isotopic, thus concordant. To obtain the desired pscm on $M$ we use the concordance metric on $\partial A_1\x[-1,0]\subset A_1$ to join the pscm $\phi_1^*(G_1)$ restricted to $(A_1-(\partial A_1\x[-1,0])$ and the pscm $\phi_2^*(G_2)$ on $A'_2$ (see figure \ref{rho2}, we set $B'_1:=B_1-(\partial B_1\x[-1,0]))$.

\begin{figure}
\psset{xunit=.4pt,yunit=.4pt,runit=.4pt}
\begin{pspicture}(200,605)(540,930)
{
\newrgbcolor{curcolor}{0.7019608 0.7019608 0.7019608}
\pscustom[linestyle=none,fillstyle=solid,fillcolor=curcolor]
{
\newpath
\moveto(431.45693,621.03706262)
\curveto(428.21644,628.79601262)(425.77814,651.36566262)(424.02237,689.85370262)
\curveto(423.17432,708.44361262)(423.17437,771.08936262)(424.02245,790.11134262)
\curveto(425.83519,830.77008262)(428.71402,855.68216262)(432.07243,859.77203262)
\lineto(432.95166,860.84275262)
\lineto(440.02272,859.84847262)
\curveto(469.50041,855.70354262)(492.88319,850.36692262)(517.60397,842.14222262)
\curveto(575.9068,822.74464262)(614.90515,793.20149262)(626.29066,759.80676262)
\curveto(640.96983,716.75139262)(606.67802,672.89802262)(536.23975,644.64701262)
\curveto(508.48214,633.51414262)(475.80112,625.12423262)(441.82125,620.40781262)
\curveto(432.1527,619.06581262)(432.28353,619.05786262)(431.45693,621.03706262)
\closepath
}
}
{
\newrgbcolor{curcolor}{0.90196079 0.90196079 0.90196079}
\pscustom[linestyle=none,fillstyle=solid,fillcolor=curcolor]
{
\newpath
\moveto(343.81168,615.83009262)
\curveto(226.93048,620.92506262)(130.54164,663.19212262)(110.13942,718.29671262)
\curveto(105.09749,731.91451262)(104.91042,747.18270262)(109.62366,760.39188262)
\curveto(113.83056,772.18201262)(120.73126,782.63779262)(131.314,793.25655262)
\curveto(168.22079,830.28895262)(240.75956,856.37059262)(325.31168,863.00928262)
\curveto(335.76062,863.82969262)(355.26748,864.75257262)(362.45427,864.76652262)
\curveto(366.50614,864.77452262)(366.81006,864.70622262)(366.37233,863.88835262)
\curveto(363.52266,858.56370262)(361.09344,835.55620262)(359.20706,796.02504262)
\curveto(358.37214,778.52848262)(358.37593,700.99581262)(359.21206,683.52504262)
\curveto(361.07378,644.64756262)(363.5388,621.45530262)(366.37183,616.16173262)
\curveto(366.81635,615.33115262)(366.41023,615.27985262)(359.95378,615.35096262)
\curveto(356.16285,615.39266262)(348.89868,615.60832262)(343.81118,615.83009262)
\closepath
}
}
{
\newrgbcolor{curcolor}{0.80000001 0.80000001 0.80000001}
\pscustom[linestyle=none,fillstyle=solid,fillcolor=curcolor]
{
\newpath
\moveto(367.27246,616.24994262)
\curveto(364.22117,621.02130262)(361.83983,642.26482262)(359.91229,681.90882262)
\curveto(359.06791,699.27542262)(358.75958,758.28520262)(359.39769,780.39837262)
\curveto(360.73775,826.83744262)(363.64837,858.06445262)(367.13245,863.38183262)
\curveto(367.95518,864.63747262)(368.0728,864.65882262)(374.1668,864.65882262)
\curveto(386.87592,864.65882262)(408.142,863.39645262)(426.02678,861.58038262)
\curveto(431.75545,860.99867262)(431.84579,860.97013262)(431.19091,859.94857262)
\curveto(427.92406,854.85252262)(425.01165,827.97869262)(423.38372,787.90882262)
\curveto(422.60676,768.78480262)(422.60388,710.83004262)(423.37872,691.90882262)
\curveto(424.95445,653.44120262)(427.71411,627.10599262)(430.81265,620.96747262)
\curveto(431.3482,619.90649262)(431.38838,619.43229262)(430.9649,619.17057262)
\curveto(429.42494,618.21882262)(398.26417,615.95021262)(379.35461,615.41317262)
\lineto(368.01334,615.09107262)
\lineto(367.27223,616.24994262)
\lineto(367.27246,616.24994262)
\closepath
}
}
{
\newrgbcolor{curcolor}{0 0 0}
\pscustom[linewidth=0.99999997,linecolor=curcolor,linestyle=dashed,dash=4.37344397 8.74688793]
{
\newpath
\moveto(367.57909469,615.00000766)
\curveto(373.18326341,615.00000581)(377.72633991,670.97491789)(377.72634006,740.02347303)
\curveto(377.72634021,807.73779881)(373.35095158,863.14458617)(367.85713607,865.00000275)
}
}
{
\newrgbcolor{curcolor}{0 0 0}
\pscustom[linewidth=1.00000001,linecolor=curcolor]
{
\newpath
\moveto(368.08260217,615.00000374)
\curveto(362.80746703,615.00000189)(358.53112495,670.97491478)(358.53112481,740.02347092)
\curveto(358.53112467,807.73779768)(362.64962421,863.14458584)(367.82088523,865.00000245)
}
}
{
\newrgbcolor{curcolor}{0 0 0}
\pscustom[linewidth=0.99999999,linecolor=curcolor]
{
\newpath
\moveto(432.85623246,618.99999872)
\curveto(427.33388679,618.99999693)(422.85714098,673.1837126)(422.85714083,740.02271496)
\curveto(422.85714069,805.57018327)(427.16864693,859.20395422)(432.58225058,860.99999749)
}
}
{
\newrgbcolor{curcolor}{0 0 0}
\pscustom[linewidth=0.69999995,linecolor=curcolor,linestyle=dashed,dash=3.11160078 6.22320153]
{
\newpath
\moveto(432.70989522,618.97727872)
\curveto(438.31406487,618.97727693)(442.85714212,673.1609926)(442.85714227,739.99999496)
\curveto(442.85714242,805.54746327)(438.48175307,859.18123422)(432.98793666,860.97727749)
}
}
{
\newrgbcolor{curcolor}{0 0 0}
\pscustom[linewidth=0.99999999,linecolor=curcolor]
{
\newpath
\moveto(629.99999936,739.99999996)
\curveto(629.99999936,670.96440741)(512.63464277,615.0000021)(367.85713836,615.0000021)
\curveto(223.07963395,615.0000021)(105.71427736,670.96440741)(105.71427736,739.99999996)
\curveto(105.71427736,809.0355925)(223.07963395,864.99999782)(367.85713836,864.99999782)
\curveto(512.63464277,864.99999782)(629.99999936,809.0355925)(629.99999936,739.99999996)
\closepath
}
}
{
\newrgbcolor{curcolor}{0 0 0}
\pscustom[linewidth=1,linecolor=curcolor]
{
\newpath
\moveto(395,820.00000262)
\lineto(395,885.00000262)
}
}
{
\newrgbcolor{curcolor}{0 0 0}
\pscustom[linestyle=none,fillstyle=solid,fillcolor=curcolor]
{
\newpath
\moveto(395,815.38400262)
\lineto(391,822.30400262)
\lineto(399,822.30400262)
\lineto(395,815.38400262)
\closepath
}
}
{
\newrgbcolor{curcolor}{0 0 0}
\pscustom[linewidth=1,linecolor=curcolor]
{
\newpath
\moveto(395,815.38400262)
\lineto(391,822.30400262)
\lineto(399,822.30400262)
\lineto(395,815.38400262)
\closepath
}
}
{
\newrgbcolor{curcolor}{0 0 0}
\pscustom[linewidth=1,linecolor=curcolor]
{
\newpath
\moveto(434,666.00000262)
\lineto(530,630.00000262)
}
}
{
\newrgbcolor{curcolor}{0 0 0}
\pscustom[linestyle=none,fillstyle=solid,fillcolor=curcolor]
{
\newpath
\moveto(429.67790452,667.62078842)
\lineto(437.56179619,668.93633092)
\lineto(434.75280866,661.4456975)
\lineto(429.67790452,667.62078842)
\closepath
}
}
{
\newrgbcolor{curcolor}{0 0 0}
\pscustom[linewidth=1,linecolor=curcolor]
{
\newpath
\moveto(429.67790452,667.62078842)
\lineto(437.56179619,668.93633092)
\lineto(434.75280866,661.4456975)
\lineto(429.67790452,667.62078842)
\closepath
}
}
\rput(240,740){{\small $G_1$ on $P_1\!\x\!B'_1$}}
\rput(530,740){{\small $G_2$ on $P_2\x B'_2$}}
\rput(380,900){Concordance metric on $P_1\!\x\!P_2\!\x\!Q\!\x\![-1,0]$}
\rput(560,620){$\partial A_1$}
\end{pspicture}
\caption{Construction of a pscm on $A_1\cup A_2$}\label{rho2}\end{figure}
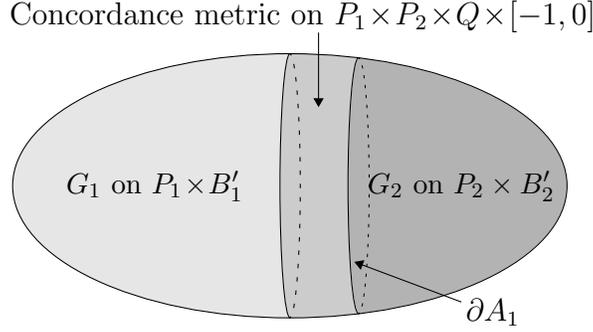

The arduousness of the proof of Thm.~ \ref{pscmprop} for $\mathscr{P}$-manifolds consisting of more $P_i$-parts merely lies in the fact that metrics on diverse submanifolds have to be chosen in a compatible way.

For simplicity we omit diffeomorphisms in the sequel. Let $M=A_1\cup\ldots\cup A_k$ be a $\mathscr{P}$-manifold. As above fix bicollar neighborhoods $\partial A_i\x[-1,1]$ of $\partial A_i\subset M$, say $\partial A_i\x\{-1\}\subset A_i$, such that $(\partial A_i\x[-1,1])\cap A_j$ is a bicollar neighborhood for $(\partial A_i)\cap A_j$ in $A_j$ for all $j$. Define a new covering by  $A'_1:=A_1$ and, for $i>1$, $A'_i:=A_i-(\cup_{j<i}\mathring{A}_j)$. Then one has $A'_i=P_i\x B'_i$ for appropriate $B'_i\subset B_i$.

We set $A^j=\cup_{i=j}^kA'_i$. For $1< j\leq k$ note that $A^j$ is modelled on $\mathbb{H}_1^n\cap\ldots\cap\mathbb{H}_{j-1}^n$ and inherits collar neighborhoods $(\partial A_i\cap A^j)\x[0,1]$ for all $i<j$. By a collar metric on $A^j$ we understand a metric which
\begin{itemize}
\item extends to a smooth metric on $M$,
\item restricts for all $i<j$ to a product on $(\partial A_i\cap A^j)\x[0,1]$ with the standard metric on the second factor.
\end{itemize}
The same notation is used for other manifolds modelled on intersections of half spaces, like $A'_j$, $B'_j$ and $Q^j_I$ (defined in the sequel). All upcoming metrics are supposed to be collar metrics.

For all $1\leq j\leq k$ and
\begin{equation}\label{folge}
I=\{i_1,\ldots,i_s\},\, 1\leq i_1<\ldots< i_s \leq j-1
\end{equation}
there exists a manifold $Q^j_{I}$ such that
\begin{equation}\label{schnitt}
\left(\bigcap_{i\in I}\partial A_i\right)\cap A'_j=P_{i_1}\x\ldots\x P_{i_s }\x P_j\x Q^j_{I}.
\end{equation}
With this notation we have $Q_{\emptyset}^i=B'_i$. Note that in particular, since $(\cap_{i=1}^{k-1}\partial A_i)$ lies in $A_k$,
\[
\left(\bigcap_{i=1}^{k-1} \partial A_i\right)=P_1\x\ldots\x P_k\x Q^k_{\{1,\ldots,k-1\}}.
\]
The manifold $Q:=Q^k_{\{1,\ldots,k-1\}}$ can be described as the 'deepest' point of $M$. Now choose pscm $g_i$ on $P_i$ for $1\leq i<k$ and a metric $h$ on $Q$. We need the above concordance argument in the following form.
\begin{lem}\label{2con} Assume that there is a $1\leq j<k$ and a pscm $G^{j+1}$ on $A^{j+1}$. One finds an $R\subset\partial B'_j$ such that $\partial A_j\cap A^{j+1}=P_j\x R$. In fact, $R=\left(\cup_{i=j+1}^k P_i\x Q^i_j\right)$, cf.~ figure \ref{rho3}. Assume further that $G^{j+1}|_{P_j\x R}=g_j\x h_j|_{R}$ for some metric $h_j$ on $B'_j$. Then there exists an extension of $G^{j+1}$ to a pscm $G^j$ on $A^j$.

\end{lem}
\begin{proof}
We agreed that $h_j$ is a collar metric for the induced collar neighborhood $R\x[-1,0]$ of $R$ in $B'_j$. One finds an $\epsilon>0$ such that $(\epsilon g_j)\x h_j$ is a pscm on $A'_j$. Since $G^{j+1}$ is a collar metric, $g_j\x h_j|_R$ and thus $(\epsilon g_j)\x h_j|_R$ are pscm. It is obvious that they are isotopic, hence concordant. Denote by $G$ the concordance metric on $(P_j\x R)\x[-1,0]$. Now we define a pscm on
\[
A^j=A^{j+1}\cup (P_j\x R\x[-1,0])  \cup (A'_j-(P_j\x R\x[-1,0]))
\]
by
\[
G^{j+1}\cup G \cup \left((\epsilon g_j)\x h_j|_{B'_j-(R\x[-1,0])}\right).
\]
\end{proof}

Let us describe our strategy how to construct a pscm on $M$ in terms of a $\mathscr{P}$-manifold with three $P_i$-parts (see figure \ref{rho3}). Recall that we chose a metric $h$ on $Q$ and pscm $g_1$ resp.~ $g_2$ on $P_1$ resp.~ $P_2$.

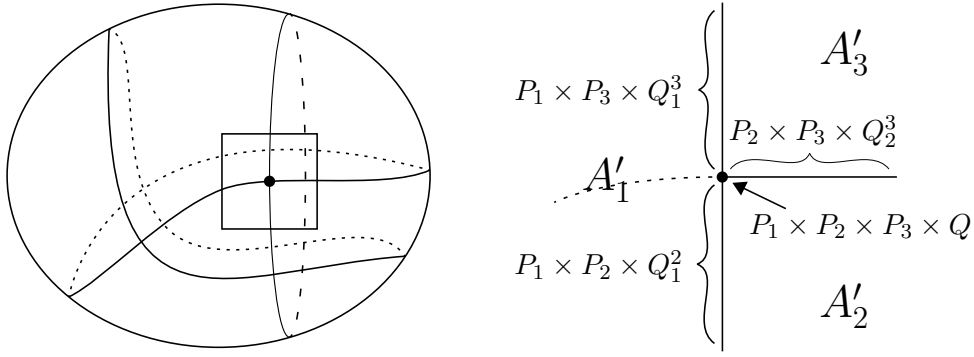
\begin{figure}
\psset{xunit=.6pt,yunit=.6pt,runit=.6pt}
\begin{pspicture}(150,735)(600,978)
\rput(620,770){{\Large $A'_2$}}
\rput(620,930){{\Large $A'_3$}}
\rput(470,850){{\Large $A'_1$}}
\rput(600,880){{\small $P_2\x P_3\x Q^3_2$}}
\rput(465,905){{\small $P_1\x P_3\x Q^3_1$}}
\rput(465,795){{\small $P_1\x P_2\x Q^2_1$}}
\rput(630,820){{\small $P_1\x P_2\x P_3\x Q$}}
{
\newrgbcolor{curcolor}{0 0 0}
\pscustom[linewidth=1,linecolor=curcolor]
{
\newpath
\moveto(543.01802,851.12708262)
\lineto(653.01802,851.12708262)
}
}
{
\newrgbcolor{curcolor}{0 0 0}
\pscustom[linewidth=1,linecolor=curcolor]
{
\newpath
\moveto(543.01802,961.12708262)
\lineto(543.01802,741.12708262)
}
}
{
\newrgbcolor{curcolor}{0 0 0}
\pscustom[linewidth=1.00000002,linecolor=curcolor,linestyle=dashed,dash=2.61823171 5.23646343]
{
\newpath
\moveto(543.0180107,851.12708127)
\curveto(498.87945387,851.30627351)(458.01183554,845.12850628)(436.96061397,835.09487933)
}
}
{
\newrgbcolor{curcolor}{0 0 0}
\pscustom[linewidth=1,linecolor=curcolor]
{
\newpath
\moveto(583.01802,831.12708262)
\lineto(554.34732,845.00000262)
}
}
{
\newrgbcolor{curcolor}{0 0 0}
\pscustom[linestyle=none,fillstyle=solid,fillcolor=curcolor]
{
\newpath
\moveto(550.19218493,847.01055194)
\lineto(558.16353068,847.59710717)
\lineto(554.67904312,840.39583323)
\lineto(550.19218493,847.01055194)
\closepath
}
}
{
\newrgbcolor{curcolor}{0 0 0}
\pscustom[linewidth=1,linecolor=curcolor]
{
\newpath
\moveto(550.19218493,847.01055194)
\lineto(558.16353068,847.59710717)
\lineto(554.67904312,840.39583323)
\lineto(550.19218493,847.01055194)
\closepath
}
}
{
\newrgbcolor{curcolor}{0 0 0}
\pscustom[linestyle=none,fillstyle=solid,fillcolor=curcolor]
{
\newpath
\moveto(546.47598402,851.09999298)
\curveto(546.47598402,849.14601333)(544.89197054,847.56199985)(542.93799089,847.56199985)
\curveto(540.98401123,847.56199985)(539.39999775,849.14601333)(539.39999775,851.09999298)
\curveto(539.39999775,853.05397264)(540.98401123,854.63798612)(542.93799089,854.63798612)
\curveto(544.89197054,854.63798612)(546.47598402,853.05397264)(546.47598402,851.09999298)
\closepath
}
}
{
\newrgbcolor{curcolor}{0 0 0}
\pscustom[linewidth=1.00000002,linecolor=curcolor]
{
\newpath
\moveto(538.01801387,846.12708132)
\curveto(521.54687027,834.12708102)(543.50839507,802.12708022)(527.03725147,796.12708007)
}
}
{
\newrgbcolor{curcolor}{0 0 0}
\pscustom[linewidth=1.00000002,linecolor=curcolor]
{
\newpath
\moveto(538.01801387,746.12707882)
\curveto(521.54687027,758.12707912)(543.50839507,790.12707992)(527.03725147,796.12708007)
}
}
{
\newrgbcolor{curcolor}{0 0 0}
\pscustom[linewidth=0.70000001,linecolor=curcolor]
{
\newpath
\moveto(270.01801159,750.12708277)
\curveto(263.11495407,750.12708126)(257.51892013,795.80260961)(257.51891995,852.14623072)
\curveto(257.51891976,907.40112067)(262.90840049,952.61305924)(269.67552802,954.12707917)
}
}
{
\newrgbcolor{curcolor}{0 0 0}
\pscustom[linewidth=1,linecolor=curcolor,linestyle=dashed,dash=2 4]
{
\newpath
\moveto(358.60755,855.43617262)
\curveto(153.69366,910.02277262)(131.00028,775.49807262)(131.00028,775.49807262)
}
}
{
\newrgbcolor{curcolor}{0 0 0}
\pscustom[linewidth=1,linecolor=curcolor,linestyle=dashed,dash=2 4]
{
\newpath
\moveto(156.38895,944.52750262)
\curveto(179.73182,935.36957262)(156.49249,854.23095262)(193.66054,818.05420262)
\curveto(229.79846,782.88009262)(321.24827,834.51184262)(343.01802,801.12708262)
}
}
{
\newrgbcolor{curcolor}{0 0 0}
\pscustom[linewidth=0.99999999,linecolor=curcolor,linestyle=dashed,dash=4.84148413 9.68296826]
{
\newpath
\moveto(270.01801469,750.12708277)
\curveto(275.62218341,750.12708126)(280.16525991,795.80260961)(280.16526006,852.14623072)
\curveto(280.16526021,907.40112067)(275.78987158,952.61305924)(270.29605607,954.12707917)
}
}
{
\newrgbcolor{curcolor}{0 0 0}
\pscustom[linewidth=1.00000003,linecolor=curcolor]
{
\newpath
\moveto(358.71622206,851.89086615)
\curveto(358.71622206,792.11118076)(299.02280314,743.65017075)(225.38723919,743.65017075)
\curveto(151.75167524,743.65017075)(92.05825633,792.11118076)(92.05825633,851.89086615)
\curveto(92.05825633,911.67055153)(151.75167524,960.13156155)(225.38723919,960.13156155)
\curveto(299.02280314,960.13156155)(358.71622206,911.67055153)(358.71622206,851.89086615)
\closepath
}
}
{
\newrgbcolor{curcolor}{0 0 0}
\pscustom[linewidth=1,linecolor=curcolor]
{
\newpath
\moveto(131.00028,775.49807262)
\curveto(152.11701,785.26941262)(191.8137,829.24002262)(227.10033,843.10182262)
\curveto(258.28761,855.35326262)(328.27173,842.07664262)(358.60755,855.43617262)
}
}
{
\newrgbcolor{curcolor}{0 0 0}
\pscustom[linewidth=1,linecolor=curcolor]
{
\newpath
\moveto(156.38895,944.52750262)
\curveto(147.21757,733.11368262)(230.70574,792.81826262)(343.01802,801.12708262)
}
}
{
\newrgbcolor{curcolor}{0 0 0}
\pscustom[linewidth=1,linecolor=curcolor]
{
\newpath
\moveto(227.51800537,878.3270874)
\lineto(287.51800537,878.3270874)
\lineto(287.51800537,818.3270874)
\lineto(227.51800537,818.3270874)
\closepath
}
}
{
\newrgbcolor{curcolor}{0 0 0}
\pscustom[linewidth=1.00000001,linecolor=curcolor]
{
\newpath
\moveto(538.01801395,956.12708132)
\curveto(523.01801455,944.12708102)(543.01801375,912.12708022)(528.01801435,906.12708007)
}
}
{
\newrgbcolor{curcolor}{0 0 0}
\pscustom[linewidth=1.00000001,linecolor=curcolor]
{
\newpath
\moveto(538.01801395,856.12707882)
\curveto(523.01801455,868.12707912)(543.01801375,900.12707992)(528.01801435,906.12708007)
}
}
{
\newrgbcolor{curcolor}{0 0 0}
\pscustom[linewidth=0.3938525,linecolor=curcolor]
{
\newpath
\moveto(648.0180187,856.12707866)
\curveto(636.0180184,871.12707806)(604.0180176,851.12707886)(598.01801745,866.12707826)
}
}
{
\newrgbcolor{curcolor}{0 0 0}
\pscustom[linewidth=0.3938525,linecolor=curcolor]
{
\newpath
\moveto(548.0180162,856.12707866)
\curveto(560.0180165,871.12707806)(592.0180173,851.12707886)(598.01801745,866.12707826)
}
}
{
\newrgbcolor{curcolor}{0 0 0}
\pscustom[linestyle=none,fillstyle=solid,fillcolor=curcolor]
{
\newpath
\moveto(261.03698622,848.40799298)
\curveto(261.03698622,846.45401333)(259.45297274,844.86999985)(257.49899309,844.86999985)
\curveto(255.54501343,844.86999985)(253.96099995,846.45401333)(253.96099995,848.40799298)
\curveto(253.96099995,850.36197264)(255.54501343,851.94598612)(257.49899309,851.94598612)
\curveto(259.45297274,851.94598612)(261.03698622,850.36197264)(261.03698622,848.40799298)
\closepath
}
}
\end{pspicture}
\caption{$\mathscr{P}$-manifold and 'deepest' point neighborhood}\label{rho3}\end{figure}

\begin{enumerate}
\item Extend $g_2\x h$ to a metric $h^3_1$ on $Q^3_1$ and $h^3_2$ on $Q^3_2$.
\item Extend $(g_1\x h^3_1)\cup (g_2\x h^3_2)$ to a metric $h_3$ on $B'_3$.
\item Choose a pscm $g_3$ on $P_3$ such that $G^3:=g_3\x h_3$ is a pscm on $A'_3$.
\item Extend $g_3\x h$ to a metric $h^2_1$ on $Q^2_1$ and $(g_1\x h^2_1)\cup(g_3\x h^3_2)$ to a metric $h_2$ on $B_2'$.
\item Apply Lemma \ref{2con} to extend $G^3$ to a pscm $G^2$ on $A'_2\cup A'_3$. Observe that Lemma \ref{2con} gives us a metric $f_1$ on $P_2\x Q^2_1$ such that $G^2|_{P_1\x (P_2\x Q^2_1)}=g_1\x f_1$.
\item Extend $(g_3\x h_1^3)\cup f_1$ to a metric $h_1$ on $B'_1$.
\item Apply Lemma \ref{2con} to extend $G^2$ to pscm a $G^1$ on $A'_1\cup A'_2\cup A'_3$.
\end{enumerate}

Turn back to the general case. For the upcoming construction it is very helpful to keep figure \ref{rho3} in mind. Let $1\leq j\leq k$ and $1\leq r< j$. One verifies that
\[
\partial A^j=\bigcup_{r=1}^{j-1}\left(\partial A_r\cap A^j\right)
\]
and
\[
\partial A_r\cap A^j=P_r\x\left(\bigcup_{i=j}^k P_i\x Q^i_r\right).
\]

We shall prove the following statement by induction over $j$, starting with $j=k$ and ending with $j=1$, from which Thm.~ \ref{pscmprop} follows immediately.

\begin{lem}\label{schw} There exists a pscm $G^j$ on $A^j$ with compatible product structures on $\partial A^j$ which means that for all $1\leq r<j$ there are metrics $f^j_r$ on $\cup_{i=j}^k P_i\x Q^i_r$ such that
\begin{equation}\label{comp2}
G^j|_{\partial A_r\cap A^j}=g_r\x f^j_r,
\end{equation}
where $g_r$ are the fixed pscm from above. \end{lem}
\begin{proof}
For the initial step, which corresponds to the steps (1)~ -~ (3) above, one has to consider $A'_k$. First we define a metric $h_k$ on $B'_k$. Observe that for $J\subsetneq\{1,\ldots,k-1\}$
\begin{equation}\label{randq}  
\partial Q_J^k=\bigcup_{\{S\subset\{1,\dots,k-1\}\, |\, J\subset S,\ |S|=|J|+1\}}P_{S-J}\x Q_S^k.
\end{equation} 
Step by step, starting with $|J|=k-2$ and ending with $J=\emptyset$, we extend the metric $h$ on the 'deepest' point $Q$ to metrics $h_J^k$ on $Q_J^k$ such that $h_J^k|_{(P_r\x Q^k_{J\cup\{r\}})\subset\partial Q_J^k}=g_r\x h_{J\cup\{r\}}^k$ for $\{r\}= I-J$.

One obtains a metric $h_k:=h^k_{\emptyset}$ on $B'_k=Q_{\emptyset}^k$
such that for all $1\leq r< k$
\begin{equation}\label{comp1}
h_k|_{(P_r\x Q^k_r)\subset\partial B_k'}= g_r\x h^k_r.
\end{equation}
Now choose a pscm $g_k$ on $P_k$ such that $G^k:=g_k\x h_k$ is a pscm on $A'_k$. Set $f^k_r:=g_k\x h^k_r$. By means of \ref{comp1} the condition \ref{comp2} is satisfied. This finishes the induction basis.

We turn to the induction step which corresponds to the steps (4)~ -~ (5) resp.~ (6)~ -~ (7) above. Let $1\leq j<k$ and $G^{j+1}$ a pscm on $A^{j+1}$ such that \ref{comp2} is satisfied. First we define a metric on
\[
\partial B'_j=\left(\bigcup_{r=1}^{j-1}P_r\x Q^j_r\right)\cup \left(\bigcup_{i=j+1}^k P_i\x Q^i_j\right)
\]
as follows. As above one finds metrics $h^j_r$ on $Q^j_r$ and we consider $g_r\x h^j_r$ on $P_r\x Q^j_r$ for all $r\leq j-1$. By induction hypothesis there is the metric $f^{j+1}_j$ on $\cup_{i=j+1}^k P_i\x Q^i_j$. Now extend
\[
\left(\bigcup_{r=1}^{j-1}g_r\x h^j_r\right)\cup f^{j+1}_j
\]
to a metric $h_j$ on $B_j'$.

Finally we apply Lemma \ref{2con} to obtain a pscm $G^j$ on $A^j$. We note that the concordance metric in Lemma \ref{2con} does not alter the $g_r$ factor on $P_r\x P_j\x Q_r^j$. Hence, for all $r\leq j-1$, there is an induced metric $f_r$ on $P_j\x Q_r^j$ such that $G^j|_{P_r\x (P_j\x Q_r^j)}=g_r\x f_r$. By means of the induction hypothesis one verifies that condition \ref{comp2} is satisfied with $f^j_r:=f_r\cup f_r^{j+1}$ on $\cup_{i=j}^kP_i\x Q_r^i$ .
\end{proof}


\providecommand{\bysame}{\leavevmode\hbox to3em{\hrulefill}\thinspace}
\providecommand{\MR}{\relax\ifhmode\unskip\space\fi MR }
\providecommand{\MRhref}[2]{%
  \href{http://www.ams.org/mathscinet-getitem?mr=#1}{#2}
}
\providecommand{\href}[2]{#2}

\end{document}